\documentclass[reqno, 11pt]{amsart}

\usepackage{
amssymb,
amsmath,
amsthm,
eucal,
dsfont,
multicol,
mathrsfs,
tikz,
graphicx,
hyperref
}
\usepackage{lipsum}
\makeatletter\renewcommand*{\eqref}[1]{%
\hyperref[{#1}]{\textup{\tagform@{\!\!\ref*{#1}}}}%
}\makeatother %add when putting this file to arXiv

\makeatletter
 
  \@addtoreset{equation}{section}
 \makeatother
\theoremstyle{plain}
\newtheorem{theorem}{Theorem}[section]
\newtheorem{lemma}[theorem]{Lemma}
\newtheorem{proposition}[theorem]{Proposition}

\theoremstyle{definition}

\newtheorem{remark}[theorem]{Remark}

\newcommand{\norm}[1]{{\|#1\|}}

\def\supp{\mathop{\mathrm{supp}}\nolimits}

\def\ac{\mathop{\mathrm{ac}}\nolimits}

\def\loc{\mathop{\mathrm{loc}}\nolimits}

\def\sgn{\mathop{\mathrm{sgn}}\nolimits}

\def\R{{\mathbb{R}}}
\def\Z{{\mathbb{Z}}}

\def\C{{\mathbb{C}}}

\def\<{{\langle}}
\def\>{{\rangle}}

\def\ep{{\varepsilon}}

\DeclareMathOperator*{\slim}{s-lim}

\title[The endpoint boundedness of wave operators]{  Counterexamples and  weak {\large \it (1,1)} estimates of wave operators for fourth-order Schr\"odinger operators in dimension three}

\author{Haruya Mizutani}
\address[H. Mizutani]{Department of Mathematics, Graduate School of Science, Osaka University, Toyonaka, Osaka 560-0043, Japan}
\email{haruya@math.sci.osaka-u.ac.jp}
\author{Zijun Wan}
\address[Z. Wan]{Department of Mathematics, School of Mathematics and Statistics, Central China Normal University, Wuhan, 430079, P.R. China}
\email{zijunwan@mails.ccnu.edu.cn}
\author{Xiaohua Yao}
\address[X. Yao]{School of Mathematics and Statistics,   Key Laboratory of Nonlinear Analysis  and Applications (Ministry of Education), Central China Normal University, Wuhan, 430079, P.R. China}
\email{yaoxiaohua@ccnu.edu.cn}
\begin{document}
\date{\today}

\begin{abstract}
	This paper is dedicated to investigating the $L^p$-bounds of wave operators $W_\pm(H,\Delta^2)$ associated with fourth-order Schr\"odinger operators $H=\Delta^2+V$ on $\mathbb{R}^3$ with  real potentials  satisfying $|V(x)|\lesssim \langle x\rangle^{-\mu}$ for some $\mu>0$.
	A recent work by Goldberg and Green \cite{GoGr21} has demonstrated that wave operators $W_\pm(H,\Delta^2)$ are bounded on $L^p(\mathbb{R}^3)$ for all $1<p<\infty$ under the condition that $\mu>9$ and zero is a regular point of $H$.
In the paper, we aim to further establish endpoint estimates for $W_\pm(H,\Delta^2)$ in two significant ways. First, we provide counterexamples to illustrate the unboundedness of $W_\pm(H,\Delta^2)$ on the endpoint spaces $L^1(\mathbb{R}^3)$ and $L^\infty(\mathbb{R}^3)$ for non-zero compactly supported potentials $V$. Second, we establish weak (1,1) estimates for the wave operators $W_\pm(H,\Delta^2)$ and their dual operators $W_\pm(H,\Delta^2)^*$ in the case where zero is a regular point and $\mu>11$. These estimates depend critically on the singular integral theory of Calder\'{o}n-Zygmund on a homogeneous space $(X,d\omega)$ with a doubling measure $d\omega$.

\end{abstract}

\maketitle
\tableofcontents

\section{Introduction}
\subsection{The main results}
Let  $H=\Delta^2+V(x)$ be the fourth-order Schr\"odinger operator on $\mathbb R^3$, where $V(x)$ is a real-valued potential satisfying
$|V(x)|\lesssim \<x\>^{-\mu}$, $x\in \mathbb R^3$ with some $\mu>0$ specified later and $\<x\>=\sqrt{1+|x|^2}$.
As $\mu>1$, it was well-known  (see  e.g. \cite{Agmon,Kuroda,ReSi}) that the {\it wave operators}
\begin{align}\label{def-wave}
W_\pm=W_\pm(H,\Delta^2) :=\slim_{t\to\pm\infty}e^{itH}e^{-it\Delta^2}
\end{align}
exist as partial isometries on  $L^2(\mathbb{R}^3)$ and are asymptotically complete.

Note that $W_\pm$ are clearly bounded on $L^2(\R^3)$. Hence it would be interesting to establish the following $L^p$-bounds of $W_\pm$ for $p\neq2$:
\begin{align}
\label{Lp-bound}\|W_\pm \phi\|_{L^p(\mathbb{R}^3)}\lesssim \|\phi\|_{L^p(\mathbb{R}^3)}.
%,\quad \|W_\pm^* \phi\|_{L^p(\mathbb{R}^3)}\lesssim \|\phi\|_{L^p(\mathbb{R}^3)}.
\end{align}
To explain the importance of these bounds, recall that $W_\pm$ satisfy the following identities
\begin{align*}
W_\pm W_\pm^* =P_{\ac}(H),\ \ W_\pm^*W_\pm=I,
\end{align*}
and {\it the intertwining property} $f(H)W_\pm=W_\pm f(\Delta^2)$,
where $f$ is  a Borel measurable function on $\mathbb{R}$. These formulas especially imply
\begin{align}
\label{intertwining_1}
f(H)P_{\ac}(H)=W_\pm f(\Delta^2)W_\pm^*.
\end{align}
 By virtue of \eqref{intertwining_1}, the $L^p$-boundedness of $W_\pm, W_\pm^*$ can immediately be used to reduce the $L^p$-$L^q$ estimates for the perturbed operator $f(H)$ to the same estimates for the free operator $f(\Delta^2)$ as follows:
\begin{align}
\label{Lp-bound of f(H)}
\|f(H)P_{\ac}(H)\|_{L^p\to L^q}\le \|W_\pm\|_{L^q\to L^q}\ \|f(\Delta^2)\|_{L^p\to L^q}\ \|W_\pm^*\|_{L^{p}\to L^{p}}.
\end{align}
For many cases, under suitable conditions on $f$, it is feasible to establish the $L^p$-$L^q$ bounds of  $f(\Delta^2)$ by Fourier multiplier methods. Thus, in order to obtain the inequality \eqref{Lp-bound of f(H)},  it is a key problem to prove the $L^p$-bounds \eqref{Lp-bound} of $W_\pm$ and $W_\pm^*$.

Recently, in the regular case (i.e., zero is neither an eigenvalue nor a resonance of $H$), Goldberg and Green \cite{GoGr21} have demonstrated that the wave operators $W_\pm$ are bounded on $L^p(\mathbb{R}^3)$ for all $1<p<\infty$ if $|V(x)|\lesssim \langle x \rangle^{-\mu}$ for some $\mu>9$ and there are no embedded positive eigenvalues in the spectrum of $H=\Delta^2+V$. Therefore, it is natural to consider whether the boundedness of $W_\pm$ holds for the endpoint cases, namely, when $p=1$ and $p=\infty$.

The following theorem provides a negative answer, showing that the wave operators $W_\pm$ are unbounded on $L^1(\mathbb{R}^3)$ and $L^\infty(\mathbb{R}^3)$ assuming that $V$ is compactly supported on $\mathbb{R}^3$.  Furthermore, weak (1,1) estimates for $W_\pm$ can be established  in the regular case, provided that $\mu>11$.

%Recently, in the regular case (i.e.  zero is neither an eigenvalue nor resonance of $H$),  Goldberg and Green \cite{GoGr21}  have proved that the wave operators $W_\pm$ are bounded on $L^p(\R^3)$ for all  $1<p<+ \infty$ if  $|V(x)|\lesssim \< x \>^{-\mu}$ for some $\mu >9$ and  there are no embedded  positive eigenvalues in the spectrum of $H=\Delta^2+V$.  Therefore,   it is natural  to ask whether the boundedness of $W_\pm $ holds or not for the  endpoint  cases $p=1$ and $\infty$.  Indeed, the following  theorem first gives an answer negatively, which shows that   the wave operators $W_\pm$  are unbounded  on  $L^1(\R^3)$ and $L^\infty(\R^3)$  even assume that $V$ is compactly supported on $\R^3$. Moreover,  the weak estimate of $W_\pm$ can be established on $L^1(\R^3)$ in the regular case with $\mu>11$.
%such that zero is regular point of $H$ and  $H$ has no embedded positive eigenvalues.

In order to state our results, we denote by $\mathbb B(X,Y)$  the space of bounded operators from $X$ to $Y$, $\mathbb B(X)=\mathbb B(X,X)$, and  by $L^{1,\infty}(\R^3)$  the weak  $L^{1}(\R^3)$.  Moreover,  we say that {\it zero is a regular point of $H=\Delta^2+V$ if  there only exists zero solution to  $H\psi=0$ in the weighted space $L_{-s}^2(\R^3)$ for all $s>{3\over 2}$}, where $L_{-s}^2(\R^3)=\langle \cdot\rangle^{s} L^2(\R^3)$.
%theorem
\begin{theorem}
	\label{theorem_2}
Let $H=\Delta^2+V(x).$ 	Suppose that $V$ is compactly supported and $V\not\equiv0$ such that zero is a regular point of $H$ and  $H$ has no embedded eigenvalue in $(0,\infty)$. Then $W_\pm, W_\pm^*\notin \mathbb B\big(L^1(\R^3)\big)\cup \mathbb B\big(L^\infty(\R^3)\big)$.
\end{theorem}

%\subsection{Main results}

\begin{theorem}
\label{theorem_1}
Let $V$ satisfy $|V(x)|\lesssim \<x\>^{-\mu}$ for  some $\mu>11$. Assume also $H$ has no embedded eigenvalue in $(0,\infty)$ and zero is a regular point of $H$. Then $W_\pm,W_\pm^*\in \mathbb{B}\big(L^1(\R^3), L^{1,\infty}(\R^3)\big)$, that is,
\begin{align}\label{Weakesti}
\big|\{x\in \R^3; \ |W_\pm f(x)|\ge\lambda\}\big|&\lesssim \frac{1}{\lambda}\int_{\R^3}|f(x)|dx,\  \lambda>0,
\end{align}
with the analogous estimate for $W_\pm^*$.
\end{theorem}

%remark
\begin{remark}
By the interpolation and the duality, Theorem \ref{theorem_1} also implies $W_\pm \in\mathbb{B}\big(L^p(\R^3)\big)$ for all $1<p<\infty$, while this is already known due to Goldberg and Green \cite{GoGr21}.
\end{remark}

Finally, we would like to emphasize that the condition of the absence of embedded positive eigenvalues is a fundamental  assumption when studying dispersive estimates and $L^p$-bounds of wave operators for higher-order Schr\"odinger operators. In fact, for any dimension $d\geq 1$, it is relatively straightforward to construct a potential function $V\in C_0^\infty(\mathbb{R}^d)$ such that $H=\Delta^2+V$ has some positive eigenvalues, as demonstrated, for instance, in \cite[Section 7.1]{FSWY}.

On the other hand, it is worth noting that Feng et al. in \cite{FSWY} have proven that $H=\Delta^2 +V$ does not have any positive eigenvalues under the assumption that the potential $V$ is bounded and satisfies the repulsive condition, meaning that $(x\cdot \nabla)V \leq 0$. Additionally, it is well-established, as demonstrated by Kato in \cite{Ka}, that the Sch\"{o}dinger operator $-\Delta +V$ has no positive eigenvalues when the potential is bounded and satisfies the condition $V(x)= o(|x|^{-1})$ as $|x|\rightarrow\infty$.
Consequently, these studies indicate that establishing the absence of positive eigenvalues for fourth-order Schr\"odinger operators is a more intricate task compared to second-order cases when dealing with bounded potential perturbations.

%We  would emphasize  that the condition on no embedded positive eigenvalues is an indispensable  assumption for higher order Schr\"odinger operators in the studies of  dispersive estimates and $L^p$-bounds of wave operators. Indeed, for any dimension $d\geq 1$,  one can easily construct some $V\in C_0^\infty(\mathbb{R}^d)$ such that $H=\Delta^2+V$ has some positive eigenvalue, see  e.g. \cite[Section 7.1]{FSWY}.  On the other hand, it should be noticed that Feng et al. in \cite{FSWY} have proved that $H=\Delta^2 +V$ does not contain any positive eigenvalues assuming that  potential $V$ is bounded and satisfies the repulsive condition (i.e. $(x\cdot \nabla)V \leq 0$). Comparatively, it was well-known by Kato \cite{Ka} that Schr\"{o}dinger operator $-\Delta +V$ has no positive eigenvalues if a bounded potential $V(x)= o(|x|^{-1})$ as $|x|$ goes to infinity. Therefore, these studies indicate that the absence of positive eigenvalues  for the fourth-order Schr\"{o}dinger operator would be more subtle than second order cases with a bounded potential perturbation $V$.

\subsection{Further  backgrounds}
For the classical Schr\"odinger operator $H=-\Delta+V(x)$, since the seminal work \cite{Yajima-JMSJ-95} of Yajima, there exists a great number of interesting works on the $L^p$-boundedness for the wave operators $W_\pm$.  More specifically,  in the space dimension $d=1$, the wave operators $W_\pm$ are bounded on  $L^p(\mathbb{R})$ for $1<p<\infty$ for both regular and zero resonance cases but in general unbounded on $L^p(\mathbb{R})$ for
$p=1,\infty$ (see e.g. \cite{ArYa,DaFa,Weder}).
 In the regular case, for dimension $d=2$ the wave operators $W_\pm$ are bounded on  $L^p(\mathbb{R}^2)$ for $1<p<\infty$ but the result of endpoint is unknown (see \cite{Yajima-CMP-99, Jensen_Yajima_2D}). For dimensions $d\geq3$, the wave operators $W_\pm$ are bounded on $L^p(\mathbb{R}^d)$ for $1\leq p\leq\infty$ in the regular case (see, for example, \cite{BeSc,Yajima-1995,Yajima-JMSJ-95}).  However, the existence of threshold resonances  shrink the range of $p$, which depends on dimension $d$ and the decay properties of zero energy eigenfunctions (see \cite{EGG,Finco_Yajima_II,Goldberg-Green-Advance,Goldberg-Green-Poincare,Jensen_Yajima_4D, Yajima_2006,Yajima_2016,Yajima_2018,Yajima_2021arxiv,Yajima_2022arxiv}).

%In the regular case, for dimension $d=2$ the wave operators $W_\pm$ are bounded on  $L^p(\mathbb{R}^2)$ for $1<p<\infty$ but the result of endpoint is unknown,  and for  dimension $d\geq3$ the wave operators $W_\pm$ are bounded on  $L^p(\mathbb{R}^d)$ for $1\leq p\leq\infty$ (see e. g.  \cite{BeSc,Jensen_Yajima_2D,Yajima-1995,Yajima-JMSJ-95,Yajima-CMP-99}). However, the existences  of threshold resonances  shrink the range of $p$, which depends on dimension $d$ and the decay properties of zero energy eigenfunctions
%(see e.g.  \cite{EGG,Finco_Yajima_II,Goldberg-Green-Advance,Goldberg-Green-Poincare,Jensen_Yajima_4D,Weder_arxiv, Yajima_2006,Yajima_2016,Yajima_2018,Yajima_2021arxiv,Yajima_2022arxiv}).

More recently, there exist several works for the  the $L^p$-boundedness of the wave operators $W_\pm$ for higher order Schr\"odinger operators $H=(-\Delta)^m+V(x)$ especially for $m=2$.  First of all,  Goldberg and Green in \cite{GoGr21} proved that for dimension $d=3$ and $m=2$, the wave operators $W_\pm$ extend to bounded operators on  $L^p(\mathbb{R}^3)$ for $1<p<\infty$  when zero is a regular point (the endpoint case is not mentioned in \cite{GoGr21}).  Then Erdo\u{g}an and Green in \cite{Erdogan-Green21,Erdogan-Green22} further showed  that as $m>1$ and $d>2m$,  $W_\pm$  are bounded on  $L^p(\mathbb{R}^d)$ for $1\leq p\leq\infty$ for certain smooth potentials $V(x)$ in the regular case. Moreover, Erdo\u{g}an, Goldberg and Green in \cite{EGG2}
also obtained that for dimension $d>4m-1$ and $\frac{2d}{d-4m+1}<p\leq\infty$, the $L^ p$ boundedness of the wave operators may fail  for compactly supported continuous potentials if the potential is not sufficiently smooth. In our previous work \cite{MiWYa}, we studied the case $d=1$ and $m=2$ and obtained that whatever zero is a regular point or a resonance of $H$, the wave operators $W_\pm$ are bounded  on $L^p(\mathbb{R})$ for $1<p<\infty$. Moreover, if in addition $V$ is compactly supported, then $W_\pm$ are also bounded from $L^1(\R)$ to $L^{1,\infty}(\R)$. On the other hand, $W_\pm$  are shown to be unbounded  on both $L^1(\mathbb{R})$ and $L^\infty(\mathbb{R})$ at least for the regular case.  More recently, Galtbayar and  Yajima \cite{GY2} have established the $L^p$-estimates of wave operator $W_\pm$ with zero resonances for the case $m=2$ and $d=4$.

%\begin{color}{purple}
In a forthcoming paper \cite{MiWYa2},  the authors  consider all the zero resonance cases for $H=\Delta^2+V$ on $\R^3$ and show that $W_\pm\in \mathbb B(L^p(\R^3))$ for all $1<p<\infty$ in the first kind resonance case.  For the second and third kind resonance cases,  it is shown that  $W_\pm\in \mathbb B(L^p(\R^3))$ for all $1<p<3$  but $W_\pm\notin \mathbb B(L^p(\R^3))$ for any $3\le p\leq\infty$.

\subsection{The ideas of the proof}
Let us explain briefly the idea of the proof.  We begin with the stationary representation of $W_-$:
$$
W_-=I-\frac{2}{\pi i}\int_0^\infty \lambda^3 R_V^+(\lambda^4)V\left(R_0^+(\lambda^4)-R_0^-(\lambda^4)\right)d\lambda,
$$
where $R_0^\pm(\lambda)=(\Delta^2-\lambda\mp i0)^{-1}$ and $R_V^\pm(\lambda)=(H-\lambda\mp i0)^{-1}$ are the free and perturbed limiting resolvents, respectively. Since the high energy part is already known to be bounded on $L^p$ for all $1\le p\le \infty$ by \cite{GoGr21}, it is enough to deal with the low energy part
$$
W_-^L:=\int_0^\infty \lambda^3 \chi(\lambda) R_V^+(\lambda^4)V\left(R_0^+(\lambda^4)-R_0^-(\lambda^4)\right)d\lambda,
$$
with $\supp\chi\subset[-\lambda_0,\lambda_0]$ and $\lambda_0\ll1$. To regard $W_-^L$ as an (singular) integral operator, we then use the asymptotic expansion of $R_V^+(\lambda^4)V$ near $\lambda=0$.
 Note that the integral kernel of $R_0^\pm(\lambda^4)$ is explicit (see \eqref{free_resolvent}). In \cite{GoGr21}, Goldberg and Green used the expansion
\begin{align}
\label{idea_0}
R_V^+(\lambda^4)V=R_0^+(\lambda^4)v\{QA_0Q+\lambda A_1+\Gamma_2(\lambda)\}v,\quad v=|V|^{1/2},
\end{align}
where $Q=I-P$, $P=\|V\|_{L^1}^{-1}\<\cdot,v\>v$,  $A_0,A_1\in \mathbb B(L^2)$,  and $\Gamma_k(\lambda)$ denotes a $\lambda$-dependent absolutely bounded operator on $L^2$ such that $\sum_{\ell=0}^k \big\||\lambda^{\ell}\partial_\lambda^\ell \Gamma_k(\lambda)|\big\|_{L^2\to L^2}\lesssim \lambda^{k},\ 0<\lambda\le\lambda_0.$

 This formula was enough for $1<p<\infty$, while this is not the case for $p=1,\infty$ not only for the unboundedness, but also for the weak {\it (1,1)} estimate. Hence, we compute the RHS of \eqref{idea_0} more precisely to obtain
\begin{align}
\label{idea_1}
R_V^+(\lambda^4)V=R_0^+(\lambda^4)v\Big\{QA_{0}Q+\lambda\left(QA_{1,0}+A_{0,1}Q\right)
+\lambda \widetilde P+\lambda^2A_2+\Gamma_3(\lambda)\Big\}v,
\end{align}
where $A_{1,0},A_{0,1},A_2\in \mathbb B(L^2)$ and $\widetilde P=cP$ with some constant $c$. To ensure this expansion make sense, we need the condition $|V(x)|\lesssim \<x\>^{-\mu}$ with $\mu>11$.

By \eqref{idea_1}, $W_-^L$ can be written as a sum of associated six integral operators. Moreover, using the explicit formula \eqref{free_resolvent} of  $R_0^\pm(\lambda^4)$ and the cancellation property $\int Qv(x)dx=0$, we can categorize such six operators into three classes (I)--(III), where (I) is associated with $QA_{0}Q,\lambda (QA_{1,0}+A_{0,1}Q)$ and $\lambda^2A_2$, (II) with $\lambda\widetilde P$ and (III) with $\Gamma_3(\lambda)$, respectively.

The operators in the class (I) can be shown to be bounded on $L^p(\R^3)$  for all $1\le p\le\infty$. Indeed, thanks to the translation invariance of $L^p$-norms and Minkowski's integral inequality  (see e.g. \eqref{K'_11}), the proof can be reduced to deal with an integral operator with the kernel bounded by
\begin{align}
\label{idea_3}
\min\big\{ \<x\>^{-1}\<y\>^{-1}\big\langle|x|\pm|y|\big\rangle^{-2}, \ \big\langle|x|\pm|y|\big\rangle^{-4}\big\}.
\end{align}
Although classical Schur's test cannot be applied to this case, separating it into three regions $|x|\sim |y|$, $|x|\gg|y|$ and $|x|\ll |y|$, we can show it is bounded on $L^p(\R^3)$ for all $1\le p\le\infty$. For the class (III), we can apply Schur's test directly to obtain the $L^p$-boundedness for all $1\le p\le\infty$. We  would emphasize  that the strong $L^1$ and $L^\infty$ boundedness for the classes (I) and (III) are necessary to achieve the unboundedness of the full operator $W_-^L$ on $L^1$ and $L^\infty$.

For the class (II),  we show that  the operator associated with $\lambda\widetilde P$ and its adjoint are bounded from $L^1(\R^3)$ to $L^{1,\infty}(\R^3)$. To explain the main idea of this result, let us consider the following model kernel
\begin{align}
\label{idea_4}
K=\frac{|x|}{|x|^4-|y|^4}=\frac{1}{2|x|(|x|^2+|y|^2)}+\frac{1}{4|x|^2(|x|+|y|)}+\frac{1}{4|x|^2(|x|-|y|)}=:\sum_{j=1}^3 K_j,
\end{align}
restricted on the region $\{(x,y)\ |\ ||x|-|y||\ge1\}$. Note that $T_{K_1},T_{K_2}\in \mathbb B(L^1(\R^3),L^{1,\infty}(\R^3))$ since $K_1,K_2$ are dominated by $ |x|^{-3}\in L^{1,\infty}(\R^3)$. To deal with $T_{K_3}$, we use the polar coordinate to rewrite $T_{K_3}f(x)$ as the following weighted 1D singular integral:
$$
T_{K_3}f(x)=\int_0^\infty \frac{g(r)}{4|x|^2(|x|-r)}\ \chi_{\{||x|-r|\ge 1\}}\ r^2dr,\quad g(r)=\int_{S^2}f(r\omega) d \omega.
$$
We then use the theory of general C-Z singular integrals on the homogeneous space to obtain that $T_{K_3}\in \mathbb B(L^1(\R^3),L^{1,\infty}(\R^3))$.

Let us emphasize that $K$ is just a model kernel and the integral kernel $K_P$ associated with $\lambda\widetilde P$ is in fact much more complicated.   Indeed, we will show the following two different expressions:
\begin{align}
K_P(x,y)
&=-\frac{1+i}{4\pi}G(x)\Big(\frac{|x|\chi_{\{||x|-|y||\ge1\}}}{|x|^4-|y|^4}\Big)G(y)+O\Big(\frac{1}{\<x\>\<y\>\big\langle|x|-|y|\big\rangle^2}\Big)\label{First-fomula}\\
&=\frac{1}{8\pi(1+i)\|V\|^2_{L^1}}\int_{\R^6}v^2(u_1)v^2(u_2) \widetilde{K}_P(x-u_1,y-u_2)du_1du_2,\label{sec-fomula}
\end{align}
where $$G(x)=\frac{|x|}{\|V\|_{L^1}}\Big(\int_{\R^3}\frac{|V|(u)}{|x-u|}du\Big),\quad \widetilde{K}_P(z,w)=\frac{-4i|z|\chi_{\{||z|-|w||\ge1\}}}{|z|^4-|w|^4}+\Psi(z,w),$$ and $T_{\Psi}\in \mathbb B(L^p)$ for all $1\le p\le\infty$. The former equality \eqref{First-fomula}  is used for proving the weak {\it (1,1)} estimate and the latter one \eqref{sec-fomula} for the unboundedness on $L^1$ and $L^\infty$. In particular, for the unboundedness, we utilize the assumption that $\supp V\subset \{|x|\le R_0\}$ with some $R_0$ and take  characteristic functions $ f_1(y)=\chi_{\{|y|\le 1\}}$ and  $f_R(y)=\chi_{\{|y|\le R\}}$ with $R\gg R_0$ to somehow estimate $\int_{\R^3}|T_{K_P}f_1|dx$  and  $|(T_{K_P}f_R)(x)|$, respectively,  then we  show that  $T_{K_P}f_1\notin L^1(\R^3)$  and   $\norm{T_{K_P}f_R}_{L^\infty(\R^3)}\to \infty$ as $R\to \infty$,
%\begin{color}{red}
%\begin{align*}
%	%\label{idea_2}
%	|T_{K_P}f_R(x_R)|\rightarrow\infty \ \text{as} \ R\rightarrow\infty,\ \text{for each}\  R+2R_0+1\leq|x_{R}|\leq R+3R_0+1,
%\end{align*}
%\end{color}
which implies the desired unboundedness of $W_\pm$ on $L^1(\R^3)$ and $L^\infty(\R^3)$.

%\subsection{The organization of the paper} The paper is organized as follows.
%
%In Section \ref{subsection_integral_operators}, we prepare necessary criterion to obtain the boundedness of integral operators associated with the wave operators $W_\pm$, including \eqref{idea_3} and \eqref{idea_4}.
%Section \ref{section_wave_operator} is mainly devoted to recall basic notions of $W_\pm$ and the resolvent expansion \eqref{idea_1}. We here also prepare some useful lemmas related with the explicit kernel of the resolvent and the cancellation property of $Q$.
%The proofs of Theorems \ref{theorem_1} and \ref{theorem_2} are given by Sections \ref{section_low} and  \ref{section_counterexample}, respectively.
%
%In Appendix \ref{appendix}, we give the detail of the proof of the resolvent expansion \eqref{idea_1}.
%%\end{color}

\subsection{Some notations}
\label{subsection_notation}
Some notations used in  the paper are listed as follows:
\begin{itemize}
\item $A\lesssim B$ (resp. $A\gtrsim B$) means $A\le CB$ (resp. $A\ge CB$) with some constant $C>0$.
\item $L^p=L^p(\R^n),L^{1,\infty}=L^{1,\infty}(\R^n)$ denote the Lebesgue and weak $L^1$ spaces, respectively.
%\item $\<f,g\>=\int_{\R^n} f\overline g$ denotes the inner product in $L^2(\R^n)$.
\item For $w\in L^1_{\loc}(\R^n)$ positive almost everywhere and $1\le p<\infty$, $L^p(w)=L^p(\R^n,wdx)$ denotes the weighted $L^p$-space with the norm
$$
\norm{f}_{L^p(w)}=\left(\int |f(x)|^pw(x)dx\right)^{1/p}.
$$
 Set $$w(E):=\int_E w(x)dx, \   \text{for each Borel subset}\  E\subset\R^n.$$  Denote $L^{1,\infty}(w)$ as  the weighted weak $L^1$ space with the quasi-norm
$$
\norm{f}_{L^{1,\infty}(w)}=\sup_{\lambda>0} \lambda w(\{x\ |\ |f(x)|>\lambda\}).
$$

%\item $\mathbb B(X,Y)$ denotes the space of bounded operators from $X$ to $Y$ and $\mathbb B(X)=\mathbb B(X,X)$.
%\item $T_K$ denotes the integral operator with the kernel $K(x,y)$, namely
%$$
%T_Kf(x)=\int_{\R^3} K(x,y)f(y)dy.
%$$
\item Let $\{\varphi_N\}_{N\in \Z}$ be a homogeneous dyadic partition of unity on $(0,\infty)$: $\varphi_0\in C_0^\infty(\R_+)$, $0\le \varphi\le1$, $\supp \varphi\subset[\frac14,1]$, $\varphi_N(\lambda)=\varphi_0(2^{-N}\lambda)$, $\supp \varphi_N\subset [2^{N-2},2^N]$  and
$$
\sum_{N\in \Z}\varphi_N(\lambda)=1,\quad\lambda>0.
$$
%\item \begin{color}{red} Write $ \mathcal{E}(\lambda)(x,y)=O_{s}(\lambda^{k}g(x,y))$ to denote that $\sum^{s}_{\ell=0}\big|\lambda^{\ell}\partial^{\ell}_\lambda\mathcal{E}(\lambda)(x,y)\big|\lesssim |\lambda^kg(x,y)|.$
%    In particular,  $ \mathcal{E}(\lambda)(x,y)=O(\lambda^{k}g(x,y))$ to denote that $|\mathcal{E}(x,y)|\lesssim |\lambda^{k}g(x,y)|.$
%    \end{color}

\end{itemize}

\section{Some integrals related with wave operators}
\label{subsection_integral_operators}

In this section, we prepare some basic criterions to the boundedness of integral operators  related with the wave operators $W_\pm$.  Throughout the paper,  we always use $T_K$ to denote the integral operator defined by the kernel $K(x,y)$:
\begin{equation}
	\label{integral}
T_Kf(x)=\int_{\R^3} K(x,y)f(y)dy.
\end{equation}
Moreover, we say that the kernel $K(x,y)$ of an operator $T_K$ is called admissible if it satisfies
\begin{align*}
\sup_{x\in \R^3} \int_{\R^3}|K(x,y)|dy+\sup_{y\in \R^3} \int_{\R^3}|K(x,y)|dx<\infty.
\end{align*}
Let's first recall  of the classical Schur test lemma:
%lemma
\begin{lemma}
\label{lemma_2_1}
$T_K\in \mathbb B\big(L^p(\R^3)\big)$ for all $1\le p\le \infty$ if its kernel $K(x,y)$ is admissible.
\end{lemma}

%We will use this lemma to study  the boundedness of operators with such kernel satisfying  $|K(x,y)|\lesssim \<x\>^{-1}\<y\>^{-1}\<|x|-|y|\>^{-\delta}$ for some $\delta\ge2$. Indeed,
Next,  the following proposition is crucial to the $L^p$-boundedness of wave operators $W_\pm$.

%lemma
\begin{proposition}
\label{prop_2_2}
Let the kernel $K(x, y)$  satisfy the following condition:
\begin{equation}
\label{kernel type-1}
|K(x,y)|\lesssim \<x\>^{-1}\<y\>^{-1}\big{\langle}|x|-|y|\big{\rangle}^{-2}, \ \  (x,y)\in \R^3\times\R^3.
\end{equation}
Then $T_{K}\in \mathbb B\big(L^1(\R^3),L^{1,\infty}(\R^3)\big)\cap \mathbb{B}\big(L^p(\R^3)\big)$ for $1< p<\infty$. That is
%\begin{equation}
\begin{align}\label{Strong-estimate}
\norm{T_Kf}_{L^p(\R^3)}&\lesssim \norm{f}_{L^p(\R^3)},\ \, 1<p<\infty,\\
\label{Weak-estimate}
\big|\{x\in \R^3; \ |(T_Kf)(x)|\ge\lambda\}\big|&\lesssim \frac{1}{\lambda}\int_{\R^3}|f(x)|dx,\ \ \lambda>0.
\end{align}
%\end{equation}
Moreover, if there exists $\delta>0$ such that $K(x,y)$ further satisfies  one of the following two conditions
\begin{align}
\label{kernel type-2}
|K(x,y)|&\lesssim \<x\>^{-1}\<y\>^{-1}\big{\langle}|x|-|y|\big{\rangle}^{-2-\delta},\\
\label{kernel type-3}
|K(x,y)|&\lesssim \min\Big\{ \<x\>^{-1}\<y\>^{-1}\big{\langle}|x|-|y|\big{\rangle}^{-2}, \big{\langle}|x|-|y|\big{\rangle}^{-3-\delta}\Big\},
\end{align}
then $T_K\in  \mathbb{B}\big(L^p(\R^3)\big)$ for all $1\le p\le\infty$.
\end{proposition}

\begin{proof}
Firstly, we decompose $K(x,y)$ as
\begin{align*}
K(x,y)
&=K(x,y)\big(\chi_{\{\frac{1}{2}|x|\le|y|\le 2|x|\}}+\chi_{\{|y|<\frac{1}{2}|x|\}}+\chi_{\{|y|>2|x|\}}\big)\\&=: K_{1}(x,y) +K_{2}(x,y)+ K_{3}(x,y),
%&=\sum_\pm \left(\chi_\pm(x)K(x,y)\chi_\pm(y)+\chi_\pm(x)K(x,y)\chi_\mp(y)\right)\\
\end{align*}
and denote $T_{K_i}$ as the integral operators associated with the kernels $K_i(x,y)$ for $i=1,2,3$. Using  \eqref{kernel type-1}, we have
\begin{align*}
\int_{\R^3}|K_1(x,y)|dy &\lesssim\frac{1}{\<x\>^2}\int_{\frac{1}{2}|x|\le|y|\le 2|x|} \big{\langle}|x|-|y|\big{\rangle}^{-2}\  dy\\
&\lesssim \frac{|x|^2}{\<x\>^2}\int_{\frac{1}{2}|x|}^{ 2|x|} \big{\langle}|x|-r\big{\rangle}^{-2}\ dr \lesssim \int_{-\infty}^{+\infty}\<r\>^{-2}\  dr\lesssim1,
\end{align*}
uniformly in $x\in \R^3$. Similarly, we also have
\begin{align*}
\int_{\R^3}|K_1(x,y)|dx &\lesssim\frac{1}{\<y\>^2}\int_{\frac{1}{2}|y|\le|x|\le 2|y|} \big{\langle}|x|-|y|\big{\rangle}^{-2} dx \lesssim1,
\end{align*}
uniformly in $y\in \R^3$.  Hence by Schur's test, we  conclude that $T_{K_1}\in  \mathbb{B}\big(L^p(\R^3)\big)$ for all $1\le p\le\infty$.

Now consider the  integral operator $T_{K_2}$.  Note that
 $$|K_2(x,y)|\lesssim  \<x\>^{-1}\<y\>^{-1}\big{\langle}|x|-|y|\big{\rangle}^{-2}\ \chi_{\{|y|\le \frac{1}{2}|x| \}}.$$
 Then for $f\in L^\infty(\R^3)$, we have
  \begin{align*}
  |T_{K_2}f(x)|&\lesssim \Big( \int_{|y|\le \frac{1}{2}|x|}\ \<x\>^{-1}\<y\>^{-1}\big{\langle}|x|-|y|\big{\rangle}^{-2}\ dy\Big)\ \|f\|_{L^\infty(\R^3)}\\
  &\lesssim \frac{1}{\<x\>^3}\Big( \int_{|y|\le \frac{1}{2}|x|}\<y\>^{-1}dy\Big)\ \|f\|_{L^\infty}\lesssim \|f\|_{L^\infty(\R^3)},
 \end{align*}
 which yields $T_{K_2}\in \mathbb{B}\big(L^\infty(\R^3)\big) $. On the other hand, if $f\in L^1(\R^3)$, then
 \begin{align}\label{T_k2}
  |T_{K_2}f(x)| \lesssim \<x\>^{-3}\Big( \int_{|y|\le \frac{1}{2}|x|}\<y\>^{-1}|f(y)|\ dy\Big)\le \<x\>^{-3}\ \|f\|_{L^1(\R^3)},
 \end{align}
which leads to $T_{K_2}\in \mathbb{B}(L^{1}, L^{1,\infty})$ due to $\<x\>^{-3}\in L^{1,\infty}(\R^3)$. By the Marcinkiewicz interpolation (see e.g.  Grafakos \cite[page 34]{Grafakos_Classical_III}), we obtain $T_{K_2}\in \mathbb{B}\big(L^1(\R^3),L^{1,\infty}(\R^3)\big)\cap \mathbb{B}\big(L^p(\R^3)\big)$ for all $1< p\le\infty$.

Next we deal with the third integral operator $T_{K_3}$.   Clearly, $T^*_{K_3}=T_{K^*_3}$  with
$$|K^*_3(x,y)|=|\overline{K_3(y,x)}| \lesssim  \<x\>^{-1}\<y\>^{-1}\big{\langle}|x|-|y|\big{\rangle}^{-2}\ \chi_{\{|y|\le \frac{1}{2}|x| \}}.$$
By the same argument as in  $T_{K_2}$, one has  $T_{K^*_3}\in \mathbb{B}\big(L^1(\R^3),L^{1,\infty}(\R^3)\big) \cap\mathbb{B}\big(L^p(\R^3)\big)$ for all $1< p\le\infty$.  Hence $T_{K_3}\in \mathbb{B}(L^p(\R^3))$ for all $1\le p<\infty$ by
the duality. Combining with the boundedness of $T_{K_j}$ for $j=1,2,3$, we conclude $T_K\in \mathbb{B}\big(L^1(\R^3),L^{1,\infty}(\R^3)\big) \cap\mathbb{B}\big(L^p(\R^3)\big)$ for all $1< p<\infty$ as desired.

Finally, we shall show $T_K\in  \mathbb{B}\big(L^p(\R^3)\big)$ for all $1\le p\le\infty$ under the conditions \eqref{kernel type-2} or \eqref{kernel type-3}. By the above argument, it suffices to show $T_{K_2}\in \mathbb{B}\big(L^1(\R^3)\big)$. If \eqref{kernel type-2} holds, then for any $f\in L^1(\R^3)$,
\begin{align*}
  \int_{\R^3}|T_{K_2}f(x)|dx &\lesssim \int_{\R^3} \Big( \int_{|y|\le \frac{1}{2}|x|} \<x\>^{-1}\<y\>^{-1}\big{\langle}|x|-|y|\big{\rangle}^{-2-\delta}|f(y)|\ dy\Big)\ dx\\
  &\lesssim \Big( \int_{\R^3}\<x\>^{-3-\delta}dx\Big)\ \|f\|_{L^1}\lesssim \|f\|_{L^1(\R^3)}.
 \end{align*}
That is, $T_{K_2}\in \mathbb{B}\big(L^1(\R^3)\big)$.  If \eqref{kernel type-3} holds, then
$$|K_2(x,y)|\lesssim\big{\langle}|x|-|y|\big{\rangle}^{-3-\delta}\ \chi_{\{|y|< \frac{1}{2}|x| \}},$$ for some $ \delta>0.$
Hence again, we can obtain from the \eqref{T_k2} that
$$ \int_{\R^3}|T_{K_2}f(x)|dx\lesssim \Big( \int_{\R^3}\<x\>^{-3-\delta}dx\Big)\Big(\int_{|y|<\frac{1}{2}|x|}|f(y)|dy\Big)\lesssim \|f\|_{L^1(\R^3)}.$$
Thus the whole proof of Proposition \ref{prop_2_2} has been finished.\end{proof}

\begin{remark}
	In  Proposition \ref{prop_2_2},  under the condition \eqref{kernel type-1}, the strong estimates \eqref{Strong-estimate} of  $T_K$ have been obtained  by Goldberg and Green \cite[Lemma 2.1]{GoGr21}   using  a different argument from one above.   We also remark that the weak estimate \eqref{Weak-estimate} of  $T_K$ seems to be new.
	\end{remark}

%\subsection{Non-classical kernels related with wave operators}

As is seen in Section \ref{section_low} below, Proposition \ref{prop_2_2} is not enough to prove Theorem \ref{theorem_1} and we  need to study some integral operators $T_K$ with kernels like $ K(x,y)=\frac{|x|}{|x|^4-|y|^{4}}$. To establish the $L^p$ boundedness of such an operator $T_K$, we will make use of the theory of Calde\'ron--Zygmund on the $A_p$-weighted spaces and on homogeneous spaces with  doubling measures. Although the proof of the following proposition is reduced to the  Calder\'on--Zygmund theory of singular integrals,  the kernel $\frac{|x|}{|x|^4-|y|^4}$  is not a standard Calder\'on--Zygmund kernel of $\R^3$, e.g. see Grafakos \cite[Page 359]{Grafakos_Classical_III}.

%lemma
\begin{proposition}
\label{prop_2_3}
Let $T_K$ be the integral operator with the  following truncated kernel
\begin{align*}
K(x,y):=\frac{|x|\,\chi_{\{||x|-|y||\ge 1\}}}{|x|^4-|y|^4},\ \  (x,y)\in \R^3\times \R^3.
\end{align*}
Then the operator $T_K,T_K^*\in \mathbb{B}\big(L^1(\R^3),L^{1,\infty}(\R^3)\big)\cap \mathbb{B}\big(L^p(\R^3)\big)$ for all $1< p<\infty$.  \end{proposition}

\begin{proof} It should be pointed out that Lemma 3.3 in \cite {GoGr21} implies $T_{K}\in \mathbb{ B}\big(L^p(\R^3)\big)$ for all $1< p<\infty$. Hence in the sequel we mainly show the weak estimate for  the endpoint case $p=1$ with only a sketch of the proof for $1<p<\infty$.  Following a similar method as of  \cite{GoGr21}, we reduce  the integral in three space dimensions to the one dimensional integral by the spherical coordinate transform.
  Let $g(s):=\int_{S^2}f(s\omega) d \omega $ for $s>0$ where $S^2$ is the unite sphere of $\R^3$.  Then
  \begin{align*}
  T_{K}f(x)=\int_{||x|-|y||\ge 1}\frac{|x|}{|x|^4-|y|^4}f(y)dy=\frac{|x|}{4}\int_{||x|-\sqrt[4]{r}|\ge 1}\frac{r^{-\frac{1}{4}}g(\sqrt[4]{r})}{|x|^4-r}dr:=\frac{|x|}{4}\ G(|x|^4),
 \end{align*}
where
$$G(s)=\int_{|\sqrt[4]{s}-\sqrt[4]{r}|\ge 1}\frac{r^{-\frac{1}{4}}g(\sqrt[4]{r})}{s-r}\ dr.$$
Note that in  \cite[Lemma 3.3]{GoGr21}, it was shown that the function $G(s)$ can be dominated by the maximal truncated Hilbert transform $\mathbb{H}^*(\widetilde{g})(s)$ and Littlewood-Hardy maximal function $\mathbb{M}(\widetilde{g})(s)$,  where  the function  $\widetilde{g}(r):=r^{-\frac{1}{4}}g(\sqrt[4]{r})$. That is,
$$|G(s)|\lesssim \mathbb{H}^*(\widetilde{g})(s)+\mathbb{M}(\widetilde{g})(s), \ \ s>0.$$
Since
$$\int_{\R^3}|T_Kf(x)|^pdx={\pi\over 4^{p+1}} \int_0^\infty |G(s)|^p s^{\frac{p-1}{4}}ds,$$
and  $|s|^{\frac{p-1}{4}}$ is $A_p-$weights for all $1<p<\infty$, by using the boundedness of $\mathbb{H}^*$ and $\mathbb{M}$ on $L^p(\R, |s|^{\frac{p-1}{4}}ds)$ (see e.g. Grafakos \cite[Chapter 7]{Grafakos_Classical_III}),  then it immediately  follows that
 \begin{align*}
  \int_{\R^3}|T_Kf(x)|^pdx&\lesssim \int_0^\infty|\mathbb{H}^*(\widetilde{g})(s)|^p s^{\frac{p-1}{4}}ds+\int_0^\infty|\mathbb{M}(\widetilde{g})(s)|^p s^{\frac{p-1}{4}}ds\\
  &\lesssim\int_0^\infty |\widetilde{g}(r)|^pr^{\frac{p-1}{4}}dr \lesssim \|f\|^p_{L^p(\R^3)},
 \end{align*}
which gives the integral operator $T_K\in \mathbb{B}\big(L^p(\R^3)\big)$ for all $1<p<\infty$.

We remark that the arguments above depend on the strong estimates of Hilbert  transforms  $\mathbb{H}^*(\widetilde{g})$ and  Littlewood-Hardy maximal function $\mathbb{M}(\widetilde{g})$ on $L^p(\R^3)$ for $1<p<\infty$, which do not directly work for $p=1$ or $\infty$ due to the failure of  strong estimates of  $\mathbb{H}^*$ and $\mathbb{M}$ on these limiting spaces. Hence in the following we will use another argument to prove $T_K\in\mathbb{ B}\big(L^1(\R^3), L^{1,\infty}(\R^3)\big)$.

Firstly,  we  decompose $K(x,y)$ as follows:
 \begin{align*}K(x,y)=\frac{\chi_{\{||x|-|y||\ge 1\}}}{2|x|(|x|^2+|y|^2)}+\frac{\chi_{\{||x|-|y||\ge 1\}}}{4|x|^2(|x|+|y|)}+\frac{\chi_{\{||x|-|y||\ge 1\}}}{4|x|^2(|x|-|y|)}=:\sum_{j=1}^3 K_j(x,y),
 \end{align*}
and write the integral operator $T_K$  into the sum $\sum_{j=1}^3 T_{K_j}$, respectively.  Let $f\in L^1(\R^3)$. Then for each $x\in \R^3$ we easily obtain that
$$|T_{K_1}f(x)|+|T_{K_2}f(x)|\lesssim |x|^{-3}\|f\|_{L^1(\R^3)}.$$
Since $|x|^{-3}\in L^{1,\infty}(\R^3)$, so it follows immediately that $T_{K_j}\in B\big(L^1(\R^3), L^{1,\infty}(\R^3)\big)$ for $j=1,2$.

Next,  it remains to show $T_{K_3}\in B\big(L^1(\R^3), L^{1,\infty}(\R^3)\big)$.  By the polar coordinate transform,
$$T_{K_3}f(x)=\int_0^\infty \frac{g(r)}{4|x|^2(|x|-r)}\ \chi_{\{||x|-r|\ge 1\}}\ r^2dr=\mathbb{W}(g_{0})(|x|),$$
where $g(r)=\int_{S^2}f(r\omega) d \omega $  and
$$\mathbb{W}(g_{0})(s):=\int_{\R}\frac{\chi_{\{|s-r|\ge 1\}}}{4s^2(s-r)}\ g_{0}(r) r^2dr, \ \ g_{0}(s)=\chi_{(0,\infty)}(s)g(s).$$

Let $d\mu(r)=r^2dr$ be  a Borel measure on the real line $\R$. Then $d\mu(r)$ is a doubling measure on $\R $ ( see e.g. Stein \cite[Page 12]{Stein} ).  In the following,  we will  regard  the integral  $\mathbb{W}(g_{0})$ as a singular integral on $L^1(\R, d\mu)$ in order to establish the weak estimate of  $T_{K_3}f$  on $L^1(\R^3)$.

In fact,  in view of the following facts:
 \begin{align*}
 \big|\{x\in \R^3; &\ |T_{K_3}f(x)|> \lambda\}\big|= \big|\{x\in \R^3; \ |\mathbb{W}(g_{0})(|x|)|> \lambda\}\big|\\&=4\pi\int_0^\infty \chi_{\{s\in\R;\ |\mathbb{W}(g_{0})(s)|>\lambda\}}\ s^2ds
= 4\pi \mu\{s\in\R^+;\ |\mathbb{W}(g_{0})(s)|>\lambda\};
  \end{align*}
  and
  $$\int_{\R} |g_{0}(s)|d\mu(s)\le \int_0^\infty \int_{S^2}|f(r\omega)|r^2d\omega dr =\|f\|_{L^1(\R^3)},$$
we can immediately conclude  that the operator $T_{K_3}\in \mathbb{B}\big(L^1(\R^3), L^{1,\infty}(\R^3)\big)$, if one has
   \begin{align}\label{weak estimates}
     \lambda\ \mu\{s\in\R;\ |\mathbb{W}(g_{0})(s)|>\lambda\}\lesssim \int_{\R} |g_{0}(s)|d\mu(s),  \ \lambda>0.
     \end{align}

 To obtain the weak estimate  \eqref{weak estimates},  we will make use of the theory of general  C-Z singular integral on the homogeneous space $(X, d\mu)$ with a doubling measure $\mu$. Indeed, in view of  conclusions  in Stein \cite[p.19, Theorem 1.3]{Stein}, it suffices to show that the integral $\mathbb{W}(f)$ on the homogeneous space $(\R,r^2dr)$ satisfies the following two conditions:
\begin{itemize}
 \item[(i)] There exist some $q>1$ and $A>0$ such that
 $$\|\mathbb{W}(f)\|_{L^q(\R, d\mu)}\le A \|f\|_{L^q(\R, d\mu)}, \  \ d\mu=r^2dr. $$
\vskip0.2cm
 \item[(ii)] The kernel $\mathcal{K}(s,r)=\frac{\chi_{\{|s-r|\ge 1\}}}{4s^2(s-r)}$ of the integral operator $\mathbb{W}(f)$, satisfies that
$$ \int_{|s-r|\ge 2\delta}|\mathcal{K}(s,r)-\mathcal{K}(s,\overline{r})|d\mu(s)\le A<\infty,$$
whenever $|r-\overline{r}|<\delta$ and $\delta>0$.
\end{itemize}

Firstly let us check the condition (i). Indeed, let   $1<q<{3\over 2}$, then
 \begin{align*}
 \int_{\R}|\mathbb{W}(f)(s)|^q d\mu(s)&=4^{-q}\int_{\R}\Big|\int_{|s-r|\ge 1}\frac{f(r)r^2}{s-r}dr\Big|s^{2-2q}ds\\
 &\lesssim \int_{\R}|f(r)r^2|^q r^{2-2q}dr=\|f\|_{L^q(\R, d\mu)}^q,
  \end{align*}
  where in the second inequality above,  we have used the weighted $L^q$ estimates of the truncated Hilbert transform
   on $L^q(\R, w(r)dr)$ with a $A_q$-weight $w(r)=|r|^{2-2q}$ due to the fact $-1<2-2q<q-1$ as $1<q<{3\over 2}$.

  Next, we come to prove the condition (ii). Let $\delta>0$ and $|r-\overline{r}|<\delta$. Then
  \begin{align*}
  \int_{|s-r|\ge 2\delta}&|\mathcal{K}(s,r)-\mathcal{K}(s,\overline{r})|d\mu(s)=
  \frac{1}{4}\int_{|s-r|\ge 2\delta}\Big|\frac{\chi_{\{|s-r|\ge 1\}}}{s-r}-\frac{\chi_{\{|s-\overline{r}|\ge 1\}}}{s-\overline{r}}\Big|ds\\
  &\lesssim\int_{|s-r|\ge 2\delta}\Big| \frac{\chi_{\{|s-r|\ge 1\}}}{s-r}-\frac{\chi_{\{|s-r|\ge 1\}}}{s-\overline{r}}\Big|ds+\int_{|s-r|\ge 2\delta}\Big|\frac{\chi_{\{|s-r|\ge 1\}}-\chi_{\{|s-\overline{r}| \ge1\}}}{s-\overline{r}}\Big|ds\\
  &:=I +II.
  \end{align*}
Note that $|r-\overline{r}|<\delta$ and $|s-r|\ge 2\delta$, which imply that $|s-\overline{r}|\ge \frac{1}{2}|s-r|$. Then
   $$I\le\int_{|s-r|\ge 2\delta}\frac{2|r-\overline{r}|}{|(s-r)(s-\overline{r})|}ds\le 2\delta\int_{|s-r|\ge 2\delta}\frac{ds}{|s-r|^2}=4,$$
   and
   \begin{align*}II&\le\int_{|s-r|\ge 2\delta}\frac{(\chi_{\{|s-\overline{r}|\ge 1/2\}}-\chi_{\{|s-\overline{r}| \ge1\}})}{|s-\overline{r}|}ds+\int_{1>|s-r|\ge 2\delta}\frac{\chi_{\{|s-\overline{r}| \ge1\}}}{|s-\overline{r}|}ds\\
   &\  \ \ \ \ \ \ \ \ \ \ \le \int _{{1\over2}\le |s-\overline{r}|<1}\frac{ds}{|s-\overline{r}|}+\int _{|s-r|<1}ds\le 2. \end{align*}
Thus the condition (ii)  holds. Hence by summarizing above all arguments  we can conclude the desired estimate \eqref{weak estimates},  and then  $T_K\in \mathbb{B}\big(L^1(\R^3), L^{1,\infty}(\R^3)\big)$.

Finally, we observe that the kernel of $T_K^*$ is given by
$$
\overline{K(y,x)}=\chi_{\{||x|-|y||\ge 1\}}\Big(\frac{|y|}{2|x|^2(|x|^2+|y|^2)}-\frac{1}{4|x|^2(|x|+|y|)}+\frac{1}{4|x|^2(|x|-|y|)}\Big).
$$
The last two terms are equal to exactly $K_2$ and $K_3$, respectively. The first term is dominated by $|x|^{-3}/4$. Hence, the same argument as above shows $T_K^*\in \mathbb{B}\big(L^1(\R^3), L^{1,\infty}(\R^3)\big)$.
\end{proof}

%%remark
%\begin{remark}
%
%\end{remark}

\section{Stationary formula and resolvent expansion at zero}
\label{section_wave_operator}
\subsection{The stationary formulas of wave operators}
\label{section_stationary}
First of all, we observe that it suffices to deal with $W_-$  since \eqref{def-wave} implies $W_+f=\overline{W_-\overline f}$.
The starting point is the following well-known stationary representation of $W_-$ (see e.g. Kuroda \cite{Kuroda}):
\begin{align}
	\label{stationary}
	W_-=I-\frac{2}{\pi i}\int_0^\infty \lambda^3 R_V^+(\lambda^4)V\left(R_0^+(\lambda^4)-R_0^-(\lambda^4)\right)d\lambda.
\end{align}
To explain the formula \eqref{stationary}, we need to introduce some notations.  Let
\begin{align*}
R_0(z)=(\Delta^2-z)^{-1},\quad
R_V(z)=(H-z)^{-1},\quad z\in \C\setminus[0,\infty),
\end{align*}
be the resolvents of $\Delta^2$ and $H=\Delta^2+V(x)$, respectively. We denote by $R^\pm_0(\lambda),R^\pm_V(\lambda)$ their boundary values (limiting resolvents) on $(0,\infty)$, namely
$$
R^\pm_0(\lambda)=\lim_{\ep \searrow 0}R_0(\lambda\pm i\ep),\quad R_V^\pm(\lambda)=\lim_{\ep \searrow 0}R_V(\lambda\pm i\ep),\quad \lambda>0.
$$
The existence of $R^\pm_0(\lambda)$ as bounded operators from $L^2_s(\R^3)$ to $L^2_{-s}(\R^3)$ with $s>1/2$ follows from the limiting absorption principle for the resolvent $(-\Delta-z)^{-1}$ of the free Schr\"odinger operator $-\Delta$ (see e.g.  Agmon \cite{Agmon}) and the following equality:
$$
R_0(z)=\frac{1}{2\sqrt z}\left[(-\Delta-\sqrt z)^{-1}-(-\Delta+\sqrt z)^{-1}\right],\quad z\in \C\setminus[0,\infty), \ \  \hbox{Im}\sqrt{z}>0.
$$
%which is obtained by the identity $\Delta^2-z=(-\Delta-\sqrt z)(-\Delta+\sqrt z)$ and the first resolvent equation.
This formula above also gives the explicit expressions of the kernels of $R_0^\pm(\lambda^4)$:
\begin{align}
\label{free_resolvent}
R_0^\pm(\lambda^4,x,y)=\frac{1}{8\pi\lambda^2|x-y|}\Big(e^{\pm i\lambda|x-y|}-e^{-\lambda|x-y|}\Big)=\frac{F_\pm(\lambda|x-y|)}{8\pi\lambda},\
\end{align}
where $x, y\in \R^3$ and $F_\pm(s)=s^{-1}(e^{\pm is}-e^{-s})$.
The existence of $R^\pm_V(\lambda)$ for $\lambda>0$ under our assumption of Theorem \ref{theorem_1} has been also already shown (see e.g. \cite{Agmon,Kuroda}).

\subsection{Resolvent asymptotic expansions near zero}
\label{subsection_resolvent}

This subsection is mainly devoted to the study of asymptotic behaviors of the resolvent $R_V^+(\lambda^4)$ at low energy $\lambda\to +0$. We also prepare some elementary  lemmas  needed  in the proof of our main theorems.

We begin with recalling the symmetric resolvent formula for $R_V^\pm(\lambda^4)$. Let $v(x)=|V(x)|^{1/2}$ and $U(x)=\sgn V(x)$, that is $U(x)=1$ if $V(x)\ge0$ and $U(x)=-1$ if $V(x)<0$. Let $M^\pm(\lambda)=U+vR_0^\pm(\lambda^4)v$ and $(M^{\pm})^{-1}(\lambda):=[M^\pm(\lambda)]^{-1}$.
%{lemma}
\begin{lemma}
\label{lemma_3_2}
For $\lambda>0$, $M^\pm(\lambda)$ is invertible on $L^2(\R^3)$ and $R_V^\pm(\lambda^4)V$ has the form
\begin{align}
\label{lemma_3_2_1}
R_V^\pm(\lambda^4)V=R_0^\pm(\lambda^4)v (M^\pm)^{-1}(\lambda)v.
\end{align}
\end{lemma}

%proof
\begin{proof}
Due to  the absence of embedded positive eigenvalue of $H$,  it was well-known  that  $M^\pm(\lambda)$ is invertible on $L^2(\R^3)$ for all $\lambda>0$ (see e.g.  Agmon \cite{Agmon} and Kuroda \cite{Kuroda}). Since $V=vUv$ and $1=U^2$, we have
\begin{align*}
R_V^\pm(\lambda^4)v
&=R_0^\pm(\lambda^4)v-R_V^\pm(\lambda^4) v UvR_0^\pm(\lambda^4)v
=R_0^\pm(\lambda^4)v\Big(1+UvR_0^\pm(\lambda^4)v\Big)^{-1}\\
&=R_0^\pm(\lambda^4)v\Big(U+vR_0^\pm(\lambda^4)v\Big)^{-1}U^{-1}.
\end{align*}
Multiplying $Uv$ from the right, we obtain the desired formula for $R_V^\pm(\lambda^4)V$.
\end{proof}

Throughout the paper, we only use $M^+(\lambda)$, so we write $M(\lambda)=M^+(\lambda)$ for simplicity.
In order to obtain the asymptotic behaviors of  $R_V^+(\lambda^4)$  near $\lambda=0$, we  need to establish  the asymptotic expansion of $M^{-1}(\lambda)$, which plays a crucial role in the paper. To this end, we introduce some notations. We say that an integral operator $T_K\in \mathbb B\big(L^2(\R^3)\big)$ with the kernel $K$ is {\it absolutely bounded} if $T_{|K|}\in \mathbb B\big(L^2(\R^3)\big)$. Let
\begin{align}\label{Q,P}
P:=\frac{\<\cdot,v\>v}{\norm{V}_{L^1}},\quad \widetilde P=\frac{8\pi}{(1+i)\norm{V}_{L^1}}P=\frac{8\pi}{(1+i)\norm{V}_{L^1}^2}\<\cdot,v\>v,\quad Q:=I-P.
\end{align}
Note that $P$ is the orthogonal projection onto the span of $v$ in $L^2(\R^3)$, {\it i.e.} $PL^2=\mathrm{\mathop{span}}\{v\}$, and $Q(v)=0$.

%lemma
\begin{lemma}
\label{lemma_3_3}
Let $H=\Delta^2+ V(x)$ with $|V(x)|\lesssim \<x\>^{-\mu}$ for $x\in \R^3$. If zero is a regular point of $H$ and $\mu>11$, then there exists $\lambda_0>0$ such that $M^{-1}(\lambda)$ satisfies the following asymptotic expansions on $L^2(\R^3)$ for $0<\lambda\le \lambda_0$:
\begin{align}
\label{lemma_3_3_1}
%\nonumber
M^{-1}(\lambda)=QA_{0}Q+\lambda\big(QA_{1,0}+A_{0,1}Q\big)
+\lambda\widetilde P+\lambda^2A_2+\Gamma_3(\lambda),
\end{align}
 where $A_0$, $A_{1,0}$, $A_{0,1}$ and $A_2$ are $\lambda$-independent bounded operators on $L^2$ and $\Gamma_3(\lambda)$ are $\lambda$-dependent bounded operators on $L^2$ such that all the operators  in the right sides of \eqref{lemma_3_3_1} are absolutely bounded. Moreover, $\Gamma_3(\lambda)$ satisfy that for $\ell=0,1,2,3$,
\begin{align}
\label{lemma_3_3_4}
\big\||\partial_\lambda^\ell \Gamma_3(\lambda)|\big\|_{L^2\to L^2}\le C_\ell \lambda^{3-\ell},\quad 0<\lambda\le\lambda_0.
\end{align}
\end{lemma}

We remark that, in the regular case ( i.e. zero is neither an eigenvalue nor a resonance of $H$),  the expansion of  $M^{-1}(\lambda)$ at zero has been obtained with different error terms  in \cite{FSY18, Erdogan-Green-Toprak, GoGr21}.  In  Lemma \ref{lemma_3_3} above, the expansion \eqref{lemma_3_3_1} contains  more specific and higher order terms  at the cost of fast decay of $V$ in order to study the endpoint estimates of wave operators $W_\pm$ here.  For reader's convenience, we give its simple proof  as  appendix in Section \ref{appendix}.  Moreover,  it should be pointed out that asymptotic expansions  of $M^{-1}(\lambda)$  were also established in the presence of zero resonance or eigenvalue in \cite{Erdogan-Green-Toprak}.

\vskip0.2cm

In the following  we give some elementary but useful lemmas.

%{lemma}
\begin{lemma}
\label{lemma_3_4}
Let $\lambda>0$ and $x,y\in \R^3$. If $F\in C^1(\R_+)$,
 then
\[
F(\lambda|x-y|)=F(\lambda |x|)-\lambda \int_0^1\big\langle y, w(x-\theta y) \big\rangle\  F'(\lambda|x-\theta y|)d\theta,
\]
where $F'(s)$ is the first order derivative of $F(s)$, $\langle\cdot,\cdot\rangle$ denotes the inner product of $\R^3$, and $w(x)=\frac{x}{|x|}$ for $x\neq0$ and $w(x)=0$ for $x=0$.

\begin{proof}
Let $G_\varepsilon(y)= F(\lambda\sqrt{\varepsilon^2+|x-y|^2}),\, \varepsilon \neq 0$.
Then $G_\varepsilon(y)\in C^{1}(\mathbb{R}^3)$  for $\varepsilon \neq 0$ and
$F(\lambda|x-y|)=\lim_{\varepsilon\rightarrow0}G_\varepsilon(y)$.
By Taylor's expansions, we have
\begin{equation}\label{lem-Taylorformula}
   \begin{split}
G_\varepsilon(y)=G_\varepsilon(0)
+\int_0^1 \sum_{|\alpha|=1}(\partial^\alpha
G_\varepsilon )(\theta y)y^\alpha d\theta.
\end{split}
\end{equation}
Observe that
$$\partial_{y_j}G_\varepsilon(y)
=\frac{-\lambda(x_j-y_j)}{(\varepsilon^2+|x-y|^2)^{\frac{1}{2}}}F'(\lambda\sqrt{\varepsilon^2+|x-y|^2}),
\ j=1,2,3.$$
Since there exists a constants $C=C(\lambda,x,y)$ such that $|(\partial_{y_i}G_\varepsilon)(\theta y)|
 \leq C (i=1,2,3)$ for $0\leq \theta\leq 1$ and $0<\varepsilon \leq 1$,
then by the  Lebesgue dominated convergence theorem, we have for $x-\theta y \neq 0$,
\begin{equation*}
   \begin{split}
\lim_{\varepsilon\rightarrow 0}\int_0^1(\partial_{y_i}G_\varepsilon)(\theta y)d\theta
=& \int_0^1 \frac{-\lambda(x_j-\theta y_j)}{|x-\theta y|}F'(\lambda|x-\theta y|)d\theta, \ j=1,2,3,
\end{split}
\end{equation*}
and $\lim_{\varepsilon\rightarrow 0}\int_0^1(\partial_{y_i}G_\varepsilon)(\theta y)d\theta=0 \ (j=1,2,3)$ for
$x-\theta y=0 $.
From Taylor expansions \eqref{lem-Taylorformula},  we obtain that
\begin{equation*}
   \begin{split}
F(\lambda|x-y|)= F(\lambda|x|)-\lambda \int_0^1F'(\lambda|x-\theta y|)\big\langle y, w(x-\theta y) \big\rangle d\theta.
\end{split}
\end{equation*}
\end{proof}
\end{lemma}

Below we apply  Lemma \ref{lemma_3_4} for the specific functions
$
F_\pm(s)=s^{-1}(e^{\pm i s}-e^{-s})
$  to establish the following formulas used later.

%lemma
\begin{lemma}
\label{lemma_projection}
Let $Q$ be the orthogonal  projection defined in \eqref{Q,P}, $\lambda>0$ and $
F_\pm(s)=s^{-1}(e^{\pm i s}-e^{-s})
$. Then
\begin{align*}
[QvR_0^\pm(\lambda^4)f](x)
=-\frac{1}{8\pi} Q\left( v(x)\int_{\R^3} \Big(\int_0^1\big\langle x, w(y-\theta x)\big\rangle\   F_\pm^{(1)}(\lambda|y-\theta x|)d\theta\Big) f(y) dy\right);
\end{align*}
and
\begin{align*}
[R_0^\pm(\lambda^4)vQf](x)
=-\frac{1}{8\pi}\int_{\R^3}\Big( \int_0^1  F_\pm^{(1)}(\lambda|x-\theta y|)\ \big\langle y, w(x-\theta y)\big\rangle d\theta\Big) v(y)(Qf)(y) dy,
\end{align*}
where $F_\pm^{(1)}(s)=s^{-2}\big((\pm is-1)e^{\pm is}+(s+1)e^{-s}\big)$ denotes the first order derivative of
$F_\pm(s).$
\end{lemma}
\begin{remark}\label{remark_lemma_projection}
The above formulas for $QvR_0^\pm(\lambda^4)f$  and  $R_0^\pm(\lambda^4)vQf$ can be written respectively as:
$$
QvR_0^\pm(\lambda^4)f=\frac{1}{8\pi} Q\Big(\int_{\R^3}h_{\ell}(\lambda,x,y)f(y) dy\Big), \ \
R_0^\pm(\lambda^4)vQf=\frac{1}{8\pi}\int_{\R^3}h_{r}(\lambda,x,y)\big(Qf\big)(y) dy,
$$
where
\begin{align*}
h_{\ell}(\lambda,x,y)&=-v(x)\int_0^1\big\langle x, w(y-\theta x)\big\rangle\   F_\pm^{(1)}(\lambda|y-\theta x|)d\theta,\\
h_{r}(\lambda,x,y)&=-v(y)\int_0^1 \big\langle y, w(x-\theta y)\big\rangle\  F_\pm^{(1)}(\lambda|x-\theta y|)  d\theta.
\end{align*}
%$$
% f_{\pm}(\lambda,x)= -v(x)\int_{\R^3} \Big(\int_0^1\big\langle x, w(y-\theta x)\big\rangle\   F_\pm^{(1)}(\lambda|y-\theta x|)d\theta\Big) f(y) dy:=\int_{\R^3}h_{\ell}(\lambda,x,y)f(y) dy.
%$$
 Moreover,
we also notice that
$$
h_{\ell}(\lambda,x,y),\ h_{r}(\lambda,x,y)=O_{x,y}(1),\quad \lambda\to +0.
$$
Here, we use  $h(\lambda,x,y)=O_{x,y}(\lambda^k)$ to denote that $|h(\lambda,x,y)|\lesssim\lambda^k$ for fixed $x,y$.
Compared with the free resolvent $|R^\pm_0(\lambda^4)(x,y)|\lesssim\lambda^{-1}$,  such a gain of one order power of $\lambda$  will be crucial to establish stronger point-wise estimates of integral kernels related to $W_\pm$ later.
\end{remark}

\begin{proof}[Proof of Lemma \ref{lemma_projection}]
 By the \eqref{free_resolvent} and applying Lemma \ref{lemma_3_4} to $F_\pm$, we obtain
\begin{align*}
R_0^\pm(\lambda^4,x,y)=\frac{F_\pm(\lambda |y-x|)}{8\pi\lambda}
=\frac{F_\pm(\lambda |y|)}{8\pi\lambda}-\frac{1}{8\pi}\int_0^1\big\langle x, w(y-\theta x)\big\rangle F_\pm'(\lambda|y-\theta x|)d\theta.
\end{align*}
Since $Q(v)=0$, then it follows that
\begin{align*}
[Q vR_0^\pm(\lambda^4)f](x)=\frac{1}{8\pi\lambda}&Q(v)\int_{\R^3} F_\pm(\lambda|y|)f(y)dy-\\
&\frac{1}{8\pi} Q \Big(v\int_{\R^3}\Big(\int_0^1\big\langle x, w(y-\theta x)\big\rangle F_\pm'(\lambda|y-\theta x|)d\theta\Big) f(y)dy\Big)\\
=-\frac{1}{8\pi} Q &\Big(v\int_{\R^3}\Big(\int_0^1\big\langle x, w(y-\theta x)\big\rangle F_\pm'(\lambda|y-\theta x|)d\theta\Big) f(y)dy\Big).
\end{align*}
For $R_0^+(\lambda^4)vQf$,  by taking
$$
R_0^\pm(\lambda^4,x,y)
=\frac{F_\pm(\lambda |x|)}{8\pi\lambda}-\frac{1}{8\pi}\int_0^1\big\langle y, w(x-\theta y)\big\rangle F_\pm'(\lambda|x-\theta y|)d\theta,
$$
the proof is analogous.
\end{proof}

 Moreover,  we also need to frequently use the following lemmas later.

%lemma
\begin{lemma}
\label{lemma_AB}
Let $F_\pm(s)=s^{-1}(e^{\pm is}-e^{-s})$, $A_{\pm}(s)=e^{\mp is}F'_{\pm}(s)$ and $B_{\pm}(s)=e^{\mp is}F_{\pm}(s)$. Then for any $\ell\in \mathbb{N}$, the following estimates hold:
\begin{align*}
&|F^{(\ell)}_{\pm}(s)|\lesssim \< s\>^{-1},\ \  s>0, \\
&|A^{(\ell)}_{\pm}(s)|+|B^{(\ell)}_{\pm}(s)|\lesssim\< s\>^{-\ell-1},\ \  s>0,
\end{align*}
where $F^{(\ell)}_{\pm}(s)$, $A^{(\ell)}_{\pm}(s)$ denote the $\ell^{\text{th}}$ order derivative of
$F^{(\ell)}_{\pm}(s)$, $A^{(\ell)}_{\pm}(s),$  respectively.
\end{lemma}

\begin{proof} We only prove the estimates of  $A_{\pm}(s)$ due to similarity. Firstly, we calculate that
$$
A_\pm(s)=s^{-2}\big((\pm is-1)+(s+1)e^{(-1\mp i)s}\big).
$$
For each $\ell\in \mathbb{N},$ it follows by Leibniz's  rule that
$$|A_\pm^{(\ell)}(s)|\lesssim s^{-\ell-2}\Big((s+1)+\sum_{k=0}^{\ell}s^k e^{-s}\Big), $$
which gives  $|A_\pm^{(\ell)}(s)|\lesssim s^{-\ell-1}$ for $s\ge 1$.  Additionally, by Taylor's expansion of $e^{(-1\mp)s}$, we obtain
$$A_\pm(s)=\sum_{k=0}^\infty (k+1-i)(-1\mp i)^{k+1}\frac{s^k}{(k+1)!},$$
which gives $A_\pm(s)\in C^\infty(\R)$.  Hence  $|A_\pm^{(\ell)}(s)|\lesssim s^{-\ell-1}$ for $s>0$ and $\ell\in \mathbb{N}$.
\end{proof}

Finally,  we record the following  well-known lemma,  e.g. see \cite[Lemma 3.8]{GoMo}.

\begin{lemma}
	\label{lemma3.7}
	Let $\alpha$ and $\beta$ satisfy $0<\alpha<n<\beta$. Then
	$$
	\int_{\R^n} \frac{1}{\<y\>^\beta|x-y|^{\alpha}} dy\lesssim \<x\>^{-\alpha}.
	$$
	\end{lemma}

%section

%section
\section{The proof  of Theorem \ref{theorem_1}}
\label{section_low}
In this section we consider the proof  of Theorem \ref{theorem_1}.   The stationary formula \eqref{stationary} of $W_-$ is  decomposed into the low and high energy parts as follows:
 fixed $\lambda_0>0$ small enough, let $\chi\in C_0^\infty(\mathbb R)$ be such that $\chi\equiv1$ on $(-\lambda_0/2,\lambda_0/2)$ and $\supp \chi\subset [-\lambda_0,\lambda_0]$. We define
\begin{align}
	\label{W^L_-}
	W^L_-&=\int_0^\infty \lambda^3 \chi(\lambda) R_V^+(\lambda^4)V\left(R_0^+(\lambda^4)-R_0^-(\lambda^4)\right)d\lambda,
\end{align}
\begin{align}
	\label{W^H_-}
	W^H_-&=\int_0^\infty \lambda^3 \Big(1-\chi(\lambda)\Big) R_V^+(\lambda^4)V\left(R_0^+(\lambda^4)-R_0^-(\lambda^4)\right)d\lambda.
\end{align}
Then
$W_-=I-\frac{2}{\pi i}\left(W_-^L+W_-^H\right)$.
In view of the decomposition, it suffices to estimate  $W^H_-$ and $W^L_-$,  separately.  Indeed, in  the work \cite[Proposition 4.1]{GoGr21},  it has been proved that high energy part  $W^H_-$ is bounded on $L^p(\R^3)$ for all  $1\le p\le \infty$  with the  decay rate $ \mu>5$ for $V(x)$.   Hence it only remains to deal with the low energy part $W_-^L$.

Now we will prove the following conclusion:

%theorem
\begin{theorem}
\label{theorem_low_1}
Under the assumption in Theorem \ref{theorem_1}, the low energy part $W_-^L$ defined by \eqref{W^L_-} satisfies the same statement as that in Theorem \ref{theorem_1}.
\end{theorem}

Throughout this section, we thus always assume that $|V(x)|\lesssim \<x\>^{-\mu}$ with $\mu>11$ and zero is a regular point of $H$.
Substituting the expansion \eqref{lemma_3_3_1} into \eqref{lemma_3_2_1},  if $0<\lambda\le \lambda_0$,  then we have
\begin{align*}
R^+(\lambda^4)V
&=R_0^+(\lambda^4)v\Big\{QA_{0}Q+\lambda\left(QA_{1,0}+A_{0,1}Q\right)
+\lambda\widetilde P+\lambda^2A_2+\Gamma_3(\lambda)\Big\}v.
\end{align*}
Hence  $W^L_-$ can be written as follows:
\begin{align}
\label{wave_operator_regular}
W^L_-=T_{K_{0}}+T_{K_{1,0}}+T_{K_{0,1}}+T_{K_{P}}+T_{K_{2}}+T_{K_{3}},
\end{align}
where  the kernels of six operators in the right side of the \eqref{wave_operator_regular}  are given by the following integrals:
\begin{equation}\label{Kernel integral}
	\begin{split}
&K_{0}(x,y)=\int_0^\infty \lambda^3\chi(\lambda)\Big(R_0^+(\lambda^4)vQA_{0}Qv[R_0^+-R_0^-](\lambda^4)\Big)(x,y)d\lambda,\\
&K_{1,0}(x,y)=\int_0^\infty \lambda^4\chi(\lambda)\Big(R_0^+(\lambda^4)vQA_{1,0}v[R_0^+-R_0^-](\lambda^4)\Big)(x,y)d\lambda,\\
&K_{0,1}(x,y)=\int_0^\infty \lambda^4\chi(\lambda)\Big(R_0^+(\lambda^4)vA_{0,1}Qv[R_0^+-R_0^-](\lambda^4)\Big)(x,y)d\lambda,\\
&K_{P}(x,y)=\int_0^\infty \lambda^{4} \chi(\lambda)\Big(R_0^+(\lambda^4)v\widetilde{P}v[R_0^+-R_0^-](\lambda^4)\Big)(x,y)d\lambda,\\
&K_{2}(x,y)=\int_0^\infty \lambda^5 \chi(\lambda)\Big(R_0^+(\lambda^4)vA_2v[R_0^+-R_0^-](\lambda^4)\Big)(x,y)d\lambda,\\
&K_{3}(x,y)=\int_0^\infty \lambda^3 \chi(\lambda)\Big(R_0^+(\lambda^4)v\Gamma_3(\lambda)v[R_0^+-R_0^-](\lambda^4)\Big)(x,y)d\lambda.
\end{split}
\end{equation}
In view of this formula \eqref{wave_operator_regular} for $W^L_-$, Theorem \ref{theorem_low_1}  follows from the corresponding boundedness of these six integral operators.
By virtue of  Lemma \ref{lemma_projection} and Remark \ref{remark_lemma_projection},
% 	operators $T_{K_{0}},T_{K_{1,0}},T_{K_{0,1}}$ are  the  same  class as $T_{K_{2}}$, since their integrands of the corresponding kernels are all dominated by $\lambda^3$  for fixed $x,y$. Taking $T_{K_{0}}$ as an example, that's to say
% $$ \lambda^3\chi(\lambda)\Big(R_0^+(\lambda^4)vQA_{0}Qv[R_0^+-R_0^-](\lambda^4)\Big)(x,y)=O_{x,y}(\lambda^3).$$ Moreover we classify $T_{K_{P}}$ into a separate class, according to
% $$\lambda^{4} \chi(\lambda)\Big(R_0^+(\lambda^4)v\widetilde{P}v[R_0^+-R_0^-](\lambda^4)\Big)(x,y)=O_{x,y}(\lambda^2).$$
% And we also classify $T_{K_3}$ into a separate class, according to
% $$\lambda^3 \chi(\lambda)\Big(R_0^+(\lambda^4)v\Gamma_3(\lambda)v[R_0^+-R_0^-](\lambda^4)\Big)(x,y)=O_{x,y}(\lambda^4).$$
% In short,
 the six operators $T_{K_j}, T_{K_P}, T_{K_{ij}}$ are classified into the following three cases:

(Class I).  $T_{K_{0}},T_{K_{1,0}},T_{K_{0,1}}$, $T_{K_{2}}$, where all integrands  can be dominated by $ C\lambda^3$  for fixed $x,y$ in their corresponding kernel integrals \eqref{Kernel integral}. (For short, we may set $O_{x,y}(\lambda^3)$ below);

 %\vskip0.2cm
(Class II). $T_{K_{P}}$ with  $O_{x,y}(\lambda^2)$;

(Class III). $T_{K_3}$ with   $O_{x,y}(\lambda^4)$.
%\end{itemize}

In particular, all the six operators above are in fact well-defined integral operators.
Note that, since $|v(x)|\lesssim \<x\>^{-\mu/2}$ with $\mu>11$, we have
$$
\norm{\<x\>^{k}vBv\<x\>^k f}_{L^1(\R^3)}\le \norm{\<x\>^kv}_{L^2}^2\norm{B}_{L^2\to L^2}\norm{f}_{L^\infty}\lesssim \norm{\<x\>^{2k}V}_{L^1(\R^3)}\norm{f}_{L^\infty(\R^3)},
$$
for all $B=QA_{0}Q, \ QA_{1,0}, \ A_{0,1}Q,\ \widetilde P,\ A_2, \ \Gamma_3(\lambda)$ and $k<(\mu-3)/2$.
Hence, in all cases, $\<x\>^kvBv\<x\>^k$ is an absolutely bounded integral operator for any $k\le 3$ at least, satisfying
\begin{align}
\label{vBv}
\int_{\R^6}\<x\>^k|(vBv)(x,y)|\<y\>^kdxdy\lesssim \norm{\<x\>^{2k}V}_{L^1(\R^3)}<\infty,
\end{align}
where we use the notation $(vBv)(x,y)=v(x)B(x,y)v(y).$

Now let us finish the proof of Theorem \ref{theorem_low_1} in the following three propositions corresponding to the three classes (I)-(III) above.

%proposition
\begin{proposition}
\label{proposition_regular_1}
Let $K\in \{K_{0},K_{1,0},K_{0,1},K_{2}\}$. Then $T_{K}\in \mathbb B(L^p)$ for all $1\le p\le \infty$.
\end{proposition}

%proof
\begin{proof}
All the kernels $K_{0},K_{1,0},K_{0,1}$ and $K_{2}$ can be written as the difference of the following two kernels
\begin{align}
\label{G}
K^\pm_{\alpha\beta}(x,y):=\int_0^\infty \lambda^{5-\alpha-\beta}\chi(\lambda)\Big(R_0^+(\lambda^4)vQ_\alpha BQ_\beta vR_0^\pm(\lambda^4)\Big)(x,y)d\lambda,
\end{align}
with some $B\in \mathbb B(L^2)$ so that $Q_\alpha BQ_\beta $ is absolutely bounded, where  we set $Q_1=Q$, $Q_0=I$ (the identity) and
\begin{align*}
(\alpha,\beta)&=\begin{cases}(1,1)&\text{for}\quad K=K_{0},\\(1,0)&\text{for}\quad K=K_{1,0},\\(0,1)&\text{for}\quad K=K_{0,1},\\(0,0)&\text{for}\quad K=K_{2}.\end{cases}
%(\alpha,\beta)=\begin{cases}(0,2)&\text{for}\quad K=K^0_{23},\\(1,0)&\text{for}\quad K=K^0_{31},\\(0,1)&\text{for}\quad K=K^0_{32}.\end{cases}
\end{align*}
 Then we shall show $T_{K^\pm_{\alpha\beta}}$ satisfies the desired assertion for all pairs $(\alpha,\beta)$ above. To this end,
we consider two cases (i) $\alpha=\beta=1$, (ii) $\beta=0$ or $\alpha=0$.

\underline{{\it Case (i)}}. By Lemma \ref{lemma_projection} and Remark \ref{remark_lemma_projection}, 	 we can rewrite $K^\pm_{11}$ as follows.

\begin{align}
\label{K_11}
K^\pm_{11}(x,y)= &\frac{1}{64\pi^2}\int_0^\infty\lambda^3 \chi(\lambda)\Big[\int_{\R^6}\Big(\int_0^1 \<u_1,w(x-\theta_1 u_1)\>F_+^{(1)}(\lambda|x-\theta_1u_1|)d\theta_1\Big)\times\nonumber\\
&(vQA_0Qv)(u_1,u_2)\Big(\int_0^1 \<u_2,w(y-\theta_2 u_2)\>F_\pm^{(1)}(\lambda|y-\theta_2u_2|)d\theta_2\Big)du_1du_2\Big]d\lambda,
\end{align}
where $F_\pm^{(1)}(s)$ is the first order derivative of  $F_\pm(s)=s^{-1}(e^{\pm is}-e^{-s})$, $\langle\cdot,\cdot\rangle$ denotes the inner product of $\R^3$, and $w(x)=\frac{x}{|x|}$ for $x\neq0$ and $w(x)=0$ for $x=0$.

By changing the order of integrals in \eqref{K_11}, then it follows that
\begin{align*}
|K^\pm_{11}(x,y)|\le\frac{1}{64\pi^2}&\int_{\R^6\times[0,1]^2}\Big(|u_1|v(u_1)|(QA_0Q)(u_1,u_2)||u_2|v(u_2)\Big)\times\\
&\Big|\int_0^\infty\lambda^3 \chi(\lambda)F_+^{(1)}(\lambda|x-\theta_1u_1|)F_\pm^{(1)}(\lambda|y-\theta_2u_2|)d\lambda\Big|dud\theta,
\end{align*}
where $(u,\theta)=(u_1,u_2,\theta_1,\theta_2)\in \R^6\times[0,1]^2$.

Let
\begin{align}\label{G_11}
G^\pm_{11}(X,Y)=\int_0^\infty\lambda^3 \chi(\lambda)F_+^{(1)}(\lambda|X|)F_\pm^{(1)}(\lambda|Y|)d\lambda, \ \ X, Y \in \R^3.
\end{align}
Then
\begin{align}\label{K'_11}|K^\pm_{11}(x,y)|\lesssim\int_{\R^6\times[0,1]^2}\Big(|u_1|\big|(vQA_0Qv)(u_1,u_2)\big||u_2|\Big)\big|G^\pm_{11}(x-\theta_1u_1,y-\theta_2u_2)\big|dud\theta.
\end{align}
Denote by $T_{G_{11}^\pm}$ the integral operator associated with $G_{11}^\pm(x,y)$. Then by \eqref{vBv} and  \eqref{K'_11}, Minkowski's inequality and the translation invariance of $L^{p}$-norm, we can reduce the $L^p$-boundedness of $T_{K_{11}^\pm}$ to  the $L^p$-boundedness of $T_{G_{11}^\pm}$ based on the following inequality:
\begin{align*}
\|T_{K_{11}^\pm}\|_{L^p\to L^p}\lesssim\big\||x|^2V\big\|_{L^1}\ \|QA_0Q\|_{L^2\to L^2}\ \|T_{G_{11}^\pm}\|_{L^p\to L^p}, \ \ 1\le p\le\infty.
\end{align*}
Indeed, to establish the $L^p$-boundedness of $T_{G_{11}^\pm}$ for all $1\le p\le \infty$,  by Proposition \ref{prop_2_2} it suffices to prove that $G^\pm_{11}(x,y)$ satisfies
the following point-wise estimate:
\begin{equation}
\label{G_11'}
|G_{11}^\pm(x,y)|\lesssim \min\big\{ \<x\>^{-1}\<y\>^{-1}\big\langle|x|\pm|y|\big\rangle^{-2}, \ \big\langle|x|\pm|y|\big\rangle^{-4}\big\}, \ \ x, \ y\in \R^3.
\end{equation}

Now we rewrite $G^\pm_{11}(x,y)$ as an oscillatory integral:
\begin{align}\label{G_11''}
G^\pm_{11}(x,y)=\int_0^\infty\lambda^3 \chi(\lambda)e^{i\lambda(|x|\pm|y|)}A_+(\lambda|x|)A_\pm(\lambda|y|)d\lambda, \ \ x,\ y\in \R^3.
\end{align}
where $$A_\pm(s):=e^{\mp is}F_\pm^{(1)}(s)=s^{-2}\big((\pm is-1)+(s+1)e^{(-1\mp i)s}\big),$$
which by Lemma \ref{lemma_AB}, satisfies the following estimates
 \begin{align}
 \label{A-estimate}
|A^{(\ell)}_{\pm}(s)|\lesssim\< s\>^{-\ell-1},\ \  s>0, \ \ell\in\mathbb{N}_0.
\end{align}
To estimate the integral \eqref{G_11''}, we decompose $\chi$ by using the dyadic partition of unity $\{\varphi_N\}$ defined in Subsection \ref{subsection_notation}, as
$$
\chi(\lambda)=\sum_{N=-\infty}^{N_0}\widetilde \chi_N(\lambda),\quad \widetilde \chi_N(\lambda):=\chi(\lambda)\varphi_N(\lambda),\quad \lambda>0,
$$
where $N_0\lesssim \log\lambda_0\lesssim-1$ since $\supp \chi\subset [-\lambda_0,\lambda_0]$.
Then we decompose
\begin{align*}
G^\pm_{11}(x,y)=\sum_{N=-\infty}^{N_0}\int_0^\infty e^{i\lambda(|x|\pm|y|)}\Psi_N(\lambda, x,y)d\lambda:=\sum_{N=-\infty}^{N_0}E^\pm_N(x,y),
\end{align*}
where $$\Psi_N(\lambda, x,y):=\lambda^3 \widetilde \chi_N(\lambda)A_+(\lambda|x|)A_\pm(\lambda|y|). $$
Note that $\supp\widetilde\chi_N\subset [2^{N-2},2^N]$ and
\begin{align}
\label{truncated estimates}
|\partial_\lambda^\ell \widetilde \chi_N(\lambda)|\lesssim 2^{-N\ell}, \ \ell\in \mathbb{N}_0.
\end{align} Hence  by Leibniz's formula, \eqref{A-estimate}  and  \eqref{truncated estimates}, we have
$$|\partial_\lambda^k\Psi_N(\lambda, x,y)|\lesssim 2^{(3-k)N} \big\langle2^N |x|\big\rangle^{-1}\big\langle2^N |y|\big\rangle^{-1}, \ \ k\in\mathbb{N}_0.
$$
Thus by $k$-times integration by parts for $E^\pm_N(x,y)$, it follows that
\begin{align}
\label{E_N estimate}
|E^\pm_N(x,y)|\lesssim 2^{(4-k)N}\big||x|\pm|y|\big|^{-k}\big\langle2^N |x|\big\rangle^{-1}\big\langle2^N |y|\big\rangle^{-1}, \ k\in\mathbb{N}_0,
\end{align}
which leads to the following estimates for $N\le N_0$:
\begin{align*}
|E^\pm_N(x,y)|\lesssim&\begin{cases}\frac{2^{2N} }{\<x\>\<y\>}&  \text{by}\ k=0 \  \text{of}  \
\eqref{E_N estimate};\\ \frac{2^N}{1+2^{2N}(|x|\pm|y|)^2}\frac{1}{\<x\>\<y\>||x|\pm|y||}&  \text{by}\ k=1,3\  \text{of}  \
\eqref{E_N estimate}; \\ \frac{2^N}{1+2^{2N}(|x|\pm|y|)^2}\frac{1}{||x|\pm|y||^3}& \text{by}\ k=3,5 \  \text{of}  \
\eqref{E_N estimate}.\end{cases}
\end{align*}
So we get that
 \begin{align}
 \label{G-ABC}
|G^\pm_{11}(x,y)|\le \sum_{N=-\infty}^{N_0}|E^\pm_N(x,y)|\lesssim &\begin{cases}\frac{1}{\<x\>\<y\>};\\ \frac{1}{\<x\>\<y\>(|x|\pm|y|)^2};\\ \frac{1}{(|x|\pm|y|)^4}.\end{cases}
\end{align}
Therefore we have
$$|G^\pm_{11}(x,y)|\lesssim\frac{1}{\<x\>\<y\>}
\lesssim\min\Big\{\frac{1}{\<x\>\<y\>\big\langle|x|\pm|y|\big\rangle^2},\ \frac{1}{\big\langle|x|\pm|y|\big\rangle^4}\Big\},$$
if $\big||x|\pm|y|\big|\le 1$.  On the other hand, if $\big||x|\pm|y|\big|\ge 1$, then it is clear from \eqref{G-ABC} again that
$$|G^\pm_{11}(x,y)|\lesssim\min\Big\{\frac{1}{\<x\>\<y\>\big\langle|x|\pm|y|\big\rangle^2},\ \frac{1}{\big\langle|x|\pm|y|\big\rangle^4}\Big\}.$$
Thus we  obtain the desired estimate \eqref{G_11'}.

\underline{{\it Case (ii)}}. Let $\alpha=0$ or $\beta=0$. As in the case (i), it similarly follows from \eqref{free_resolvent} and Lemma \ref{lemma_projection} that
\begin{align}\label{K_ab}
|K^\pm_{\alpha\beta}(x,y)|\lesssim \int_{\R^6\times[0,1]^2}\big(|u_1|^\alpha\ |(vQ_{\alpha} BQ_{\beta} v)(u_1,u_2)|\ |u_2|^\beta \big)|G^\pm_{\alpha\beta}(x-\theta_1u_1,y-\theta_2u_2)|dud\theta,
\end{align}
where $(\alpha,\beta)=(1,0), (0,1), (0,0)$,  and
\begin{align*}
&G^\pm_{\alpha\beta}(X,Y)=\int_0^\infty\lambda^{5-\alpha-\beta} \chi(\lambda)F_+^{(\alpha)}(\lambda|X|)F_\pm^{(\beta)}(\lambda|Y|)d\lambda, \ \ X, \ Y \in\R^3.
\end{align*}
Then by using the same arguments as above,  we can obtain the same estimate \eqref{G_11'} as for $G^\pm_{\alpha\beta}$ and then the same $L^p$-boundedness of $T_{K^\pm_{\alpha\beta}}$ for all $1\le p\le \infty$. Hence this completes the proof of Proposition \ref{proposition_regular_1}.
\end{proof}

Next we consider the operator $T_{K_{3}}$ in the class (III).
%proposition
\begin{proposition}
\label{proposition_regular_2}
The operator $T_{K_{3}}$ satisfies the same statement as that in Proposition \ref{proposition_regular_1}.
\end{proposition}

%proof
\begin{proof}
We show that $|K_{3}(x,y)|\lesssim\<x\>^{-1}\<y\>^{-1}\big\langle|x|-|y|\big\rangle^{-2-\delta}$
for some $\delta>0$, which together with Lemmas \ref{lemma_2_1} and Proposition \ref{prop_2_2}, implies the desired assertion. The proof is more involved than in the previous case since $\Gamma_3(\lambda)$ depends on $\lambda$.

As before based on the free resolvent formula \ref{free_resolvent}, we can write that
\begin{align*}
K_{3}(x,y)
&=\int_0^\infty \lambda^3 \chi(\lambda)\Big(R_0^+(\lambda^4)v\Gamma_3(\lambda)v[R_0^+-R_0^-](\lambda^4)\Big)(x,y)d\lambda,\\
&=\int_0^\infty \lambda^4 \chi(\lambda)\left(\int_{\R^6}F_+(\lambda|x-u_1|)\widetilde{\Gamma}(\lambda,u_1,u_2)[F_+-F_-](\lambda|y-u_2|)du_1du_2\right)d\lambda,\\
&:=\big(K_3^+(x,y)-K_3^-(x,y)\big).
\end{align*}
where we set $\widetilde \Gamma(\lambda,u_1,u_2)=\frac{1}{64\pi^2\lambda^3}(v\Gamma_3(\lambda)v)(u_1,u_2)$ for $\lambda>0$.
%and recall (see Lemma \ref{lemma_FF}) that
%\begin{align*}
%f_{00}(\lambda,x-u_1,y-u_2)
%=-\sum_{\pm}\left(e^{i\lambda(|x-u_1|\pm |y-u_2|)}+ie^{-\lambda(|x-u_1|\pm i|y-u_2|)}\right).
%\end{align*}

Let $\Phi^\pm(x,y, u_1,u_2)=(|x-u_1|-|x|)\pm (|y-u_2|-|y|)$. Then
\begin{align*}
	K^\pm_{3}(x,y)&=\int_0^\infty  e^{i\lambda(|x|\pm|y|)}\lambda^4 \chi(\lambda)b^\pm(\lambda,x,y)d\lambda,
\end{align*}
where
\begin{align*}
b^\pm(\lambda,x,y)&=\int_{\R^6}e^{i\lambda \Phi^\pm(x,y,u_1,u_2)}B_+(\lambda|x-u_1|)\widetilde \Gamma(\lambda,u_1,u_2)B_\pm(\lambda|y-u_2|)du_1du_2,
\end{align*}
and
 $$B_\pm(s)=e^{\mp is}F_\pm(s)=s^{-1}(1-e^{(-1\mp i) s}).$$
Firstly,  using Leibniz formula, \eqref{lemma_3_3_4},  Lemma \ref{lemma_AB} and  Lemma \ref{lemma3.7}, it follows that
\begin{align}
\label{proposition_regular_2_proof_1}
|\partial_\lambda^{\ell}b^\pm(\lambda,x,y)|&\lesssim \lambda^{-\ell-2}\ \Big(\int_{\R^3}\frac{\<u_1\>^{2\ell} |V|(u_1)}{|x-u_1|^2}du_1\Big)^{\frac{1}{2}}\Big(\int_{\R^3}\frac{\<u_2\>^{2\ell} |V|(u_2)}{|y-u_2|^2}du_2\Big)^{\frac{1}{2}}\nonumber\\
&\lesssim \lambda^{-\ell-2} \ \<x\>^{-1}\<y\>^{-1},
\end{align}
for $0<\lambda\lesssim 1$, $x,y\in \R^3$ and $\ell=0,1,2,3$. Next to deal with $K^\pm_3$, we decompose $\chi$ by using the dyadic partition of unity $\{\varphi_N\}$ defined in Subsection \ref{subsection_notation}, as
$$
\chi(\lambda)=\sum_{N=-\infty}^{N_0}\widetilde \chi_N(\lambda),\quad \widetilde \chi_N(\lambda):=\chi(\lambda)\varphi_N(\lambda),\quad \lambda>0,
$$
where $N_0\lesssim \log\lambda_0\lesssim-1$, $\supp\widetilde\chi_N\subset [2^{N-2},2^N]$ and $|\partial_\lambda^\ell \widetilde \chi_N(\lambda)|\le C_\ell 2^{-N\ell}$ for all $\ell \in \mathbb{N}_0$.
Let $K^{\pm}_{3,N}$ be given by $K^{\pm}_{3}$ with $\chi$ replaced by $\widetilde \chi_N$ and decompose $K^\pm_{3}$ as
$$
K^\pm_{3}=\sum_{N\le N_0}K^{\pm}_{3,N}.
$$
Since $\lambda\sim 2^N$ on $\supp \widetilde \chi_N$, we know by  \eqref{proposition_regular_2_proof_1} that
$$
|K^{\pm}_{3,N}(x,y)|\lesssim 2^{2N}\<x\>^{-1}\<y\>^{-1}\int_{\supp \widetilde \chi_N}d\lambda\lesssim 2^{3N}\<x\>^{-1}\<y\>^{-1},\quad x,y\in \R^3.
$$
In particular, if $\big||x|\pm|y|\big|\le1$,  then
$$
|K^{\pm}_{3,N}(x,y)|\lesssim 2^{3N}\<x\>^{-1}\<y\>^{-1}\big\langle|x|\pm|y|\big\rangle^{-3}.
$$
On the other hand, if  $\big||x|\pm|y|\big|>1$,  then we obtain by integrating by parts that
\begin{align*}
K^{\pm}_{3,N}(x,y)=\frac{i}{(|x|\pm |y|)^3}\int_0^\infty  e^{i\lambda(|x|\pm |y|)}\partial_\lambda^3\Big[\lambda^4\widetilde \chi_N(\lambda)b^\pm(\lambda,x,y)\Big]d\lambda.
\end{align*}
Then \eqref{truncated estimates}, \eqref{proposition_regular_2_proof_1} and the support property of $\widetilde \chi_N$ imply
\begin{align*}
|K^{\pm}_{3,N}(x,y)|\lesssim \<x\>^{-1}\<y\>^{-1}\big\langle|x|\pm|y|\big\rangle^{-3}\ 2^{-N}\int_{2^{N-2}}^{2^N}d\lambda\lesssim \<x\>^{-1}\<y\>^{-1}\big\langle|x|\pm|y|\big\rangle^{-3},
\end{align*}
as $\big||x|\pm|y|\big|>1$. Therefore, $K^{\pm}_{3,N}(x,y)$ satisfies
\begin{align*}
|K^{\pm}_{3,N}(x,y)|&\lesssim \<x\>^{-1}\<y\>^{-1}\min\big\{2^{3N},\big\langle|x|\pm|y|\big\rangle^{-3}\big\}\\
&\lesssim 2^{3N(1-\theta)} \<x\>^{-1}\<y\>^{-1}\big\langle|x|\pm|y|\big\rangle^{-3\theta},\quad \theta\in [0,1].
\end{align*}
 In particular,  for instance, taking $\theta=5/6$, then we obtain
\begin{align*}
|K^{\pm}_{3}(x,y)|\lesssim \<x\>^{-1}\<y\>^{-1}\big\langle|x|\pm|y|\big\rangle^{-5/2}\sum_{N\le N_0}2^{N/2}\lesssim \<x\>^{-1}\<y\>^{-1}\big\langle|x|\pm|y|\big\rangle^{-5/2}.
\end{align*}
Therefore the desired result follows by Lemma \ref{lemma_2_1} and Proposition \ref{prop_2_2}.
\end{proof}

Finally, we deal with the class (II), namely  the operator $T_{K_{P}}$. First recall that
$$\widetilde{P}=\frac{8\pi}{(1+i)\|V\|_{L^1}}P,\ \ \ \ \ P=\frac{1}{\|V\|_{L^1}}\<\cdot,v\>v.$$

%proposition
\begin{proposition}
\label{proposition_regular_4}
Let $1<p<\infty$. Then $T_{K_{P}}\in \mathbb B\big(L^1(\R^3),L^{1,\infty}(\R^3)\big)\cap \mathbb B\big(L^p(\R^3)\big).$
\end{proposition}

%remark
\begin{remark}
It will be proved in Section \ref{section_counterexample} that $T_{K_{P}}\notin\mathbb B\big(L^\infty(\R^3)\big)\cup B\big(L^1(\R^3)\big)$.  \end{remark}

%proof
\begin{proof}[Proof of Proposition \ref{proposition_regular_4}]

 By using  \eqref{free_resolvent} we first calculate that
\begin{align*}
K_{P}(x,y)
&=\frac{8\pi}{(1+i)\|V\|_{L^1}}\int_0^\infty \lambda^4\chi(\lambda)\Big(R_0^+(\lambda^4)v P v[R_0^+-R_0^-](\lambda^4)\Big)(x,y)d\lambda\\
&=\frac{1}{8\pi(1+i)\|V\|_{L^1}}\int_0^\infty\lambda^2 \chi(\lambda)\\
&\ \ \ \ \ \ \ \ \times\Big(\int_{\R^6}F_+(\lambda|x-u_1|)(vPv)(u_1, u_2)(F_+-F_-)(\lambda|y-u_2|)du_1du_2\Big)d\lambda.
\end{align*}
where $F_\pm(s)=s^{-1}(e^{\pm is}-e^{-s})$ and  $(F_+-F_-)(s)=s^{-1}(e^{is}-e^{-is})$.

Note that
$$(vPv)(u_1,u_2)=\frac{v^2(u_1)v^2(u_2)}{\|V\|_{L^1}}, \ \ (u_1, u_2)\in \R^6.$$
Hence we can rewrite $K_P(x,y)$ as:
\begin{align}
\label{K_p}
K_{P}(x,y)=\frac{1}{8\pi(1+i)\|V\|^2_{L^1}}\int_0^\infty \chi(\lambda)&\Big[\int_{\R^6}\frac{v^2(u_1)v^2(u_2)}{|x-u_1||y-u_2|}\Big( e^{i\lambda|x-u_1|}-e^{-\lambda|x-u_1|}\Big) \nonumber\\
&\times\Big( e^{i\lambda|y-u_2|}-e^{-i\lambda|y-u_2|}\Big)du_1du_2\Big]d\lambda.
%=:\frac{1+i}{8\pi^2\|V\|^2_{L^1}}&\int_{\R^6}v(u_1)^2v(u_2)^2 \widetilde{K}_p(x-u_1, y-u_2) du_1du_2.
\end{align}
Let $z=x-u_1$ and $w=y-u_2$. Then
$$( e^{i\lambda |z|}-e^{-\lambda |z|})( e^{i\lambda |w|}-e^{-i\lambda |w|})=e^{i\lambda(|z|+|w|)}-e^{i\lambda(|z|-|w|)}-e^{-\lambda(|z|-i|w|)}+e^{-\lambda(|z|+i|w|)}.$$
So we can decompose $K_P(x,y)$ as follows:
\begin{align}\label{K_p'}
K_P(x,y)=\frac{1}{8\pi(1+i)\|V\|^2_{L^1}}\Big(K_{P}^{1}(x,y)-K_{P}^{2}(x,y)-K_{P}^{3}(x,y)+K_{P}^{4}(x,y)\Big),\end{align}
where
\begin{align*}
K_{P}^1(x,y)=\int_0^\infty \chi(\lambda)\Big(\int_{\R^6}\frac{v^2(u_1)v^2(u_2)}{|x-u_1||y-u_2|}
e^{i\lambda(|x-u_1|+|y-u_2|)}du_1du_2\Big)d\lambda;
\end{align*}
\begin{align*}
K_{P}^2(x,y)=\int_0^\infty \chi(\lambda)\Big(\int_{\R^6}\frac{v^2(u_1)v^2(u_2)}{|x-u_1||y-u_2|}
e^{i\lambda(|x-u_1|-|y-u_2|)}du_1du_2\Big)d\lambda;
\end{align*}
\begin{align*}
K_{P}^3(x,y)=\int_0^\infty \chi(\lambda)\Big(\int_{\R^6}\frac{v^2(u_1)v^2(u_2)}{|x-u_1||y-u_2|}
e^{-\lambda(|x-u_1|-i|y-u_2|)}du_1du_2\Big)d\lambda;
\end{align*}
\begin{align*}
K_{P}^4(x,y)=\int_0^\infty \chi(\lambda)\Big(\int_{\R^6}\frac{v^2(u_1)v^2(u_2)}{|x-u_1||y-u_2|}
e^{-\lambda(|x-u_1|+i|y-u_2|)}du_1du_2\Big)d\lambda.
\end{align*}

In the following, we will estimate these kernels $K_{P}^j(x,y)$ ($j=1,2,3,4$) case by case. We only deal with the $K_P^1(x,y)$ due to similarity. For this end, let
\begin{align}
\label{psi_1}
\psi_1(\lambda,x, y):=\int_{\R^6}\frac{v^2(u_1)v^2(u_2)}{|x-u_1||y-u_2|}
e^{i\lambda\big[(|x-u_1|-|x|)+(|y-u_2|-|y|)\big]}du_1du_2.
\end{align}
%Similarly, define  $\Phi_j(\lambda,x, y)$ ($j=2,3,4$) for $K_P^j(x,y)$ ($j=2,3,4$).
Then we obtain
$$
K_P^1(x,y)=\int_0^\infty e^{i\lambda(|x|+|y|)}\chi(\lambda)\psi_1(\lambda,x,y)d\lambda.$$
 By integration by parts it follows that
 \begin{align}
 \label{K_p1}
 K_P^1(x,y)&=\frac{1}{i(|x|+|y|)}\Big[-\psi_1(0,x,y)-\int_0^\infty  e^{i\lambda(|x|+|y|)}\partial_\lambda\big(\chi\psi_1\big)\Big]d\lambda\nonumber\\
 &=-\frac{\psi_1(0,x,y)}{i(|x|+|y|)}-\frac{\partial_\lambda\psi_1(0,x,y)}{(|x|+|y|)^2}-\frac{1}{(|x|+|y|)^2}\int_0^\infty e^{i\lambda(|x|+|y|)}
 \partial^2_\lambda\big(\chi\psi_1\big)d\lambda.
 \end{align}
By using   \eqref{psi_1}, Lemma \ref{lemma3.7} and the decay condition of potential $V$, we have
\begin{align*}
|\psi_1(\lambda,x,y)|&+|\partial_\lambda\psi_1(\lambda,x,y)|+\Big|\int_0^\infty e^{i\lambda(|x|+|y|)}
 \partial^2_\lambda\big(\chi\psi_1\big)d\lambda\Big|\\
 &\lesssim \Big(\int_{\R^3}\frac{\<u_1\>^2v^2(u_1)}{|x-u_1|}du_1\Big)\Big(\int_{\R^3}\frac{\<u_2\>^2v^2(u_2)}{|y-u_2|}du_2\Big)\lesssim\frac{1}{\<x\>\<y\>}.
\end{align*}
Therefore  \eqref{K_p1} implies that
$$K_P^1(x,y)=\frac{i}{(|x|+|y|)}\Big(\int_{\R^6}\frac{v^2(u_1)v^2(u_2)}{|x-u_1||y-u_2|}du_1du_2\Big)+O\Big(\frac{1}{\<x\>\<y\>(|x|+|y|)^2}\Big),$$
where we use $h(x,y)=O\big(g(x,y)\big)$ to denote $|h(x,y)|\lesssim|g(x,y)|.$
Similarly, we  obtain that
\begin{align*}
K_P^2(x,y)=\frac{i}{(|x|-|y|)}\Big(\int_{\R^6}\frac{v^2(u_1)v^2(u_2)}{|x-u_1||y-u_2|}du_1du_2\Big)+O\Big(\frac{1}{\<x\>\<y\>(|x|-|y|)^2}\Big);
\end{align*}
\begin{align*}
K_P^3(x,y)=\frac{1}{|x|-i|y|}\Big(\int_{\R^6}\frac{v^2(u_1)v^2(u_2)}{|x-u_1||y-u_2|}du_1du_2\Big)+O\Big(\frac{1}{\<x\>\<y\>(|x|+|y|)^2}\Big);
\end{align*}
\begin{align*}
K_P^4(x,y)=\frac{1}{|x|+i|y|}\Big(\int_{\R^6}\frac{v^2(u_1)v^2(u_2)}{|x-u_1||y-u_2|}du_1du_2\Big)+O\Big(\frac{1}{\<x\>\<y\>(|x|+|y|)^2}\Big).
\end{align*}
Therefore by \eqref{K_p'} it follows that
$$K_P(x,y)=-\frac{1+i}{4\pi\|V\|^2_{L^1}}\Big(\int_{\R^6}\frac{v^2(u_1)v^2(u_2)}{|x-u_1||y-u_2|}du_1du_2\Big)\frac{|x|^2|y|}{|x|^4-|y|^4}
+O\Big(\frac{1}{\<x\>\<y\>(|x|-|y|)^2}\Big).$$
 By  \eqref{K_p} and Lemma \ref{lemma3.7},  we also have
$$|K_P(x,y)|\lesssim \int_{\R^6}\frac{v^2(u_1)v^2(u_2)}{|x-u_1||y-u_2|}du_1du_2\lesssim \frac{1}{\<x\>\<y\>},  \ \text{for all}\ (x,y)\in \R^3\times\R^3.
$$
Hence we can finally write $K_P(x,y)$ into the following form:
\begin{align}
\label{K_p'''}
K_P(x,y)=-\frac{(1+i)}{4\pi}G(x)\Big(\frac{|x|\chi_{\{||x|-|y||\ge1\}}}{|x|^4-|y|^4}\Big)G(y)
+O\Big(\frac{1}{\<x\>\<y\>\big\langle|x|-|y|\big\rangle^2}\Big),
\end{align}
where
$$G(x)=\frac{|x|}{\|V\|_{L^1}}\Big(\int_{\R^3}\frac{|V|(u)}{|x-u|}du\Big).$$
Note that
$|G(x)|\lesssim |x|\<x\>^{-1}<\infty$  by  Lemma \ref{lemma3.7}. Then  Propositions \ref{prop_2_2} and \ref{prop_2_3} imply that
$T_{K_P},T_{K_P}^*\in \mathbb{B}(L^1,L^{1,\infty})\cap\mathbb{B}(L^p)$ for all $1<p<\infty$.
\end{proof}

In one word, putting Propositions \ref{proposition_regular_1}--\ref{proposition_regular_4} all together, we have finished the proof of Theorem \ref{theorem_low_1}.

%remark
\begin{remark}
\label{remark_regular_4}
Although the expression of $K_P(x,y)$ in  \eqref{K_p'''} is suitable to show the weak $L^1$-boundedness (i.e. $T_{K_P}\in \mathbb{B}(L^1,L^{1,\infty})$), however it is ineffective to disprove the $L^1$-$L^1$ and $L^\infty$-$L^\infty$ boundedness of $ T_{K_P}$. This is because the second part of  \eqref{K_p'''} just represents a kernel form satisfying  weak $L^1$-estimate but lacks specificity.  In Section \ref{section_counterexample} we will employ alternative formula for $K_P(x,y)$  to show $T_{K_{P}}\notin\mathbb B\big(L^\infty(\R^3)\big)\cup \mathbb B\big(L^1(\R^3)\big)$ assuming  that $V$ has compact support.
%More specifically, we will establish the following integral expression in Section \ref{section_counterexample} :
%\begin{align}
%\label{K_P4}
%K_P(x,y)=\frac{1}{8\pi(1+i)\|V\|^2_{L^1}}\int_{\R^6}v^2(u_1)v^2(u_2) \widetilde{K}_P(x-u_1,y-u_2)du_1du_2,
%\end{align}
%where
%$$\widetilde{K}_P(z,w)=\frac{-4i|z|\chi_{\{||z|-|w||\ge1\}}}{|z|^4-|w|^4}+\Psi(z,w),$$
%and $T_\Psi\in \mathbb{B}(L^p)$ for all $1\le p\le\infty$. Note that the fact of the strong boundedness of $T_\Psi$ at $p=1,\infty$ is crucial, and makes us to reduce the endpoint unboundedness of $T_{K_P}$ to the unboundedness  of operator $T_{\mathbb{G}}$ associated with the following kernel:
%$$\mathbb{G}(x,y)=\frac{-1-i}{4\pi\|V\|^2_{L^1}}\int_{\R^6}v^2(u_1)v^2(u_2) \frac{|x-u_1|\chi_{\{||x-u_1|-|y-u_2||\ge1\}}}{|x-u_1|^4-|y-u_2|^4}du_1du_2.$$
%Indeed, we can prove that $T_{\mathbb{G}}$ is unbounded on both $L^1(\R^3)$ and $L^\infty(\R^3)$ assume that $V$ has compact support, which immediately deduces that  $T_{K_P}\notin\mathbb B\big(L^\infty(\R^3)\big)\cup \mathbb B\big(L^1(\R^3)\big)$ at the same condition of $V$ (see Section \ref{section_counterexample} for more details).
%
%Finally, we remark that the strong estimate of $T_{\widetilde{K}_P}$ implies the strong estimate of $T_{K_P}$ by using Minkowski's inequality and the \eqref{K_P4}, but not for the weak estimate of $T_{K_P}$ in view of the failure of Minkowski's inequality on $L^{1, \infty}(\R^3)$. This actually explains reasons why we need the formula \eqref{K_p'''} of $K_P(x,y)$ to obtain $T_{K_P}\in \mathbb{B}(L^1,L^{1,\infty})$.
\end{remark}

%section
\section{The proof of Theorem \ref{theorem_2}}
\label{section_counterexample}
  This section is   devoted  to showing  Theorem \ref{theorem_2}. Throughout the section, we assume that $V\not\equiv 0$, $\supp V \subset B(0, R_0)$ for some $R_0 > 0$, zero is  a regular point of $H$ and $H$ has no embedded eigenvalues in $(0, \infty)$, where $B(0, R)=\{x\in \R^3\ |x|\le R\}$. 

Recall that $W_-=I-\frac{2}{\pi i}(W_-^L+W_-^H )$. Except for $T_{K_P}$, all the other terms in $W_-^L$ in the right side of \eqref{wave_operator_regular} and the high-energy part $W_-^H$ are bounded on $L^p(\R^3)$ for all $1\le p\le \infty$ by Propositions \ref{proposition_regular_1}-\ref{proposition_regular_2} and \cite[Proposition 4.1]{GoGr21}. Theorem \ref{theorem_2} thus follows from the following proposition. 

\begin{proposition}
\label{proposition_example_1}
Let $f_R=\chi_{B(0,R)}$. Then $\norm{T_{K_P}f_R}_{L^\infty(\R^3)}\to \infty$ as $R\to \infty$, and $T_{K_P}f_1\notin L^1(\R^3)$. As a consequence, $T_{K_P}$ is neither bounded on $L^\infty(\R^3)$ nor on $L^1(\R^3)$.
\end{proposition}

To prove Proposition \ref{proposition_example_1}, we begin with the following lemma which gives another expression of $K_P(x,y)$.

\begin{lemma}
\label{K_p5}
Let $K_P(x,y)$ be the kernel of the operator $T_{K_p}$ defined in  \eqref{wave_operator_regular}.  Then
\begin{align*}K_P(x,y)=\mathbb{G}(x,y)+\mathbb{F}(x,y),
\end{align*}
where
\begin{align*}
&\mathbb{G}(x,y)=\frac{-1-i}{4\pi\|V\|^2_{L^1}}\int_{\R^6}|V(u_1)V(u_2)| \frac{|x-u_1|\chi_{\{||x-u_1|-|y-u_2||\ge1\}}}{|x-u_1|^4-|y-u_2|^4}du_1du_2,\\
&\mathbb{F}(x,y)=\frac{1}{8\pi(1+i)\|V\|^2_{L^1}}\int_{\R^6}|V(u_1)V(u_2)| \Psi(x-u_1, y-u_2)du_1du_2,
%\label{Psi(z,w)}
%&\Psi(z,w):=\int_0^\infty \chi(\lambda)\Big( \frac{e^{i\lambda|z|}-e^{-\lambda|z|}}{|z|}\Big)\Big( \frac{e^{i\lambda|w|}-e^{-i\lambda|w|}}{|w|}\Big)d\lambda+\frac{4i|z|\chi_{\{||z|-|w||\ge1\}}}{|z|^4-|w|^4}.
\end{align*}
and   $\Psi(z,w)$  is an admissible  kernel on $\R^3\times\R^3$ such that $T_\Psi$ is bounded on $L^p(\R^3)$ for all $1\le p\le\infty.$   As a consequence,  $T_{\mathbb{F}}  \in \mathbb{B}\big(L^p(\R^3)\big)$ for each $1\le p\le\infty$.
\end{lemma}

\begin{proof}
Recall that $v=\sqrt{|V|}$. By \eqref{K_p}, we can write
\begin{align*}
K_{P}(x,y)=\frac{1}{8\pi(1+i)\|V\|^2_{L^1}}\int_{\R^6}|V(u_1)V(u_2)|\widetilde{K}_P(x-u_1,y-u_2) du_1du_2,
\end{align*}
where
\begin{align}\label{K_p6}
\widetilde{K}_P(z,w)=\int_0^\infty \chi(\lambda)\Big( \frac{e^{i\lambda|z|}-e^{-\lambda|z|}}{|z|}\Big)\Big( \frac{e^{i\lambda|w|}-e^{-i\lambda|w|}}{|w|}\Big)d\lambda.
\end{align}
We set
\begin{align}\label{Psi1(z,w)}
\Psi(z,w):=\widetilde{K}_P(z,w)+\frac{4i|z|\chi_{\{||z|-|w||\ge1\}}}{|z|^4-|w|^4},
\end{align}
so that  $K_P(x,y)=\mathbb{G}(x,y)+\mathbb{F}(x,y)$ as expressed  above.
If $T_{\Psi}\in \mathbb B(L^p(\R^3))$ for all $1\le p\le \infty$, then Minkowski's integral inequality and the invariance of $L^p$-norm under the translation yield
 $$\|T_\mathbb{F}\|_{L^p \to L^p}\le\frac{1}{8\sqrt{2}\pi}\|T_\Psi\|_{L^p \to L^p}.$$
 %which  immediately gives that  $T_\mathbb{F}\in \mathbb{B}(L^p)$ for all $1\le p\le \infty$.
By virtue of Schur's test,  it thus suffices to show that $\Psi$ is an admissible  kernel on $\R^3\times\R^3$, that is,
 \begin{equation}
 \label{Schur_condition}\sup_{z\in \R^3}\int_{\R^3}|\Psi(z,w)|dw+\sup_{w\in \R^3}\int_{\R^3}|\Psi(z,w)|dz<\infty.
 \end{equation}
To this end, we write $\Psi=\Psi_1+\Psi_2$, where
$$
\Psi_1(z,w)=\widetilde{K}_P(z,w)\chi_{\{||z|-|w||<1\}},\quad \Psi_2(z,w)=\left(\widetilde{K}_P(z,w)+\frac{4i|z|}{|z|^4-|w|^4}\right)\chi_{\{||z|-|w||\ge1\}}.
$$

We first deal with $\Psi_1$. Since
$$|F_\pm(s)|=\Big|\frac{e^{\pm is}-e^{-s}}{s}\Big|\lesssim \min\Big\{1, \frac{1}{s}\Big\},$$
it follows from \eqref{K_p6} that
\begin{align}\label{K_p7}|\widetilde{K}_P(z,w)|\lesssim\int_0^\infty \lambda^2\chi(\lambda)|F_+(\lambda|z|)||(F_+-F_-)(\lambda|w|)|d\lambda\lesssim \min\Big\{1, \frac{1}{|z|}, \frac{1}{|w|}, \frac{1}{|z||w|}\Big\}.
\end{align}
Using the bound $|\widetilde{K}_P(z,w)|\lesssim1$, we obtain
$$
\sup_{|z|\le1}\int_{\mathbb{R}^3}|\Psi_1(z,w)|dw\lesssim \sup_{|z|\le1}\int_{||z|-|w||<1}dw<\infty.
$$
When $|z|\ge1$, using the bound $|\widetilde{K}_P(z,w)|\lesssim |z|^{-1}|w|^{-1}$, we have
\begin{align*}
\sup_{|z|\ge1}\int_{\mathbb{R}^3}|\Psi_1(z,w)|dw\lesssim \sup_{|z|\ge1}\left(\frac{1}{|z|}\int_{||z|-|w||<1}\frac{1}{|w|}dw\right)\lesssim\sup_{|z|\ge1}\left(\frac{1}{|z|}\int^{|z|+1}_{|z|-1}r dr\right)<\infty.
\end{align*}
Thus,
\begin{align}
\label{Schur_condition_1}
\sup_{z\in \R^3}\int_{\mathbb{R}^3}|\Psi_1(z,w)|dw<\infty.
\end{align}
The same argument also shows
\begin{align}
\label{Schur_condition_2}
\sup_{w\in \R^3}\int_{\R^3}|\Psi_1(z,w)|dz<\infty.
\end{align}

To deal with $\Psi_2$, integrating by parts in \eqref{K_p6} yields
\begin{align*}
%\label{K_p8}
\widetilde{K}_P(z,w)=&\frac{1}{|z||w|}\int_0^\infty\big(e^{i\lambda|z|}-e^{-\lambda|z|} \big)\big( e^{i\lambda|w|}-e^{-i\lambda|w|}\big)\chi(\lambda)d\lambda\nonumber\\
=&\frac{1}{|z||w|}\Big(-\frac{1}{i(|z|+|w|)}+\frac{1}{i(|z|-|w|)}+\frac{1}{|z|+i|w|}-\frac{1}{|z|-i|w|}\Big)\nonumber\\
&+\frac{1}{|z||w|}\int_0^\infty \Big(-\frac{e^{i\lambda(|z|+|w|)}}{i(|z|+|w|)}+\frac{e^{i\lambda(|z|-|w|)}}{i(|z|-|w|)}
+\frac{e^{-\lambda(|z|+i|w|)}}{|z|+i|w|}-\frac{e^{-\lambda(|z|-i|w|)}}{|z|-i|w|}\Big)\chi'(\lambda)d\lambda.
%=:&\frac{-4i|z|}{|z|^4-|w|^4}+\overline{K}_P(z,w)
\end{align*}
Since
$$
\frac{1}{|z||w|}\Big(-\frac{1}{i(|z|+|w|)}+\frac{1}{i(|z|-|w|)}+\frac{1}{|z|+i|w|}-\frac{1}{|z|-i|w|}\Big)=\frac{-4i|z|}{|z|^4-|w|^4},
$$
we find
\begin{align}
\label{Psi_2}
\Psi_2(z,w)=\frac{\chi_{\{||z|-|w||\ge1\}}}{|z||w|}\int_0^\infty \Big(-\frac{e^{i\lambda(|z|+|w|)}}{i(|z|+|w|)}+\frac{e^{i\lambda(|z|-|w|)}}{i(|z|-|w|)}+\frac{e^{-\lambda(|z|+i|w|)}}{|z|+i|w|}
-\frac{e^{-\lambda(|z|-i|w|)}}{|z|-i|w|}\Big)\chi'(\lambda)d\lambda.
\end{align}
Using this expression we shall show that
\begin{align}
\label{K_p9}
|\Psi_2(z,w)|\le C_N\begin{cases}|z|^{-1}|w|^{-1}\<|z|-|w|\>^{-N}&\text{for all $(z,w)\in \supp \Psi_2$},\\ \<z\>^{-N} & \text{if $|w|\le1/2$},\\ \<w\>^{-N}&\text{if $|z|\le1/2$},\end{cases}
\end{align}
which implies
\begin{align*}
&\sup_{z\in \R^3}\int_{\R^3}|\Psi_2(z,w)|dw\\
&\le \sup_{|z|\le1/2}\int_{\R^3}|\Psi_2(z,w)|dw+\sup_{|z|>1/2}\left(\int_{|w|\le1/2}+\int_{|w|>1/2}\right)|\Psi_2(z,w)|dw\\
&\lesssim \sup_{|z|\le1/2}\int_{\R^3}\<w\>^{-N}dw+\sup_{|z|>1/2}\left(\int_{|w|\le 1/2}1dw+
\int_{|w|>1/2,\atop ||z|-|w||\ge1}|z|^{-1}|w|^{-1}\<|z|-|w|\>^{-N}dw\right)<\infty
%&\lesssim1+\int_{||z|-r|\ge1}\<|z|-r\>^{-N}dr
\end{align*}
and similarly
$$
\sup_{w\in \R^3}\int_{\R^3}|\Psi_2(z,w)|dz<\infty.
$$
These two bounds, \eqref{Schur_condition_1} and \eqref{Schur_condition_2} imply \eqref{Schur_condition}.

It remains to show \eqref{K_p9}. To prove the first estimate in \eqref{K_p9}, we observe that, since $\chi'$ is compactly supported in $(0,\infty)$, if we integrate by parts the integral in \eqref{Psi_2}, then the boundary terms at $\lambda=0,\infty$ vanish identically. Taking  into account this fact and the bounds $$||z|\pm|w||\ge ||z|-|w||\ge1,\quad ||z|\pm i|w||\ge ||z|-|w||\ge1,  \,\,\,\,\,\,  \text{on}  \supp \Psi_2, $$  we make use of integration by parts $N$ times to obtain
$$
|\Psi_2(z,w)|\le C_N|z|^{-1}|w|^{-1}||z|-|w||^{-N}\le C_N|z|^{-1}|w|^{-1}\<|z|-|w|\>^{-N}.
$$

For the second estimate in  \eqref{K_p9}, using the formula
$$
\frac{e^{\lambda(a+b)}}{a+b}-\frac{e^{\lambda(a-b)}}{a-b}%=e^{\lambda a}\left(\frac{e^{\lambda b}-1}{a+b}-\frac{e^{-\lambda b}-1}{a-b}\right)+e^{\lambda a}\left(\frac{1}{a+b}-\frac{1}{a-b}\right)
%=e^{\lambda a}\left(\frac{e^{\lambda b}-1}{a+b}-\frac{e^{-\lambda b}-1}{a-b}-\frac{2b}{a^2-b^2}\right)
=be^{\lambda a}\left(\frac{\lambda}{a+b}\frac{e^{\lambda b}-1}{\lambda b}-\frac{\lambda}{a-b}\frac{e^{-\lambda b}-1}{\lambda b}-\frac{2}{a^2-b^2}\right)
$$
with $(a,b)=(i|z|,-i|w|)$ or $(-|z|,i|w|)$, we rewrite the integrand of $\Psi_2$   as
\begin{align*}
&\frac{\chi'(\lambda)}{|z||w|}\Big(\frac{e^{\lambda(i|z|-i|w|)}}{i|z|-i|w|}-\frac{e^{\lambda(i|z|+i|w|)}}{i|z|+i|w|}
+\frac{e^{\lambda(-|z|+i|w|)}}{-|z|+i|w|}-\frac{e^{\lambda(-|z|-i|w|)}}{-|z|-i|w|}\Big)\\
&=\frac{e^{i\lambda|z|}}{|z|}\left(\frac{\lambda \chi'(\lambda)}{i|z|-i|w|}\frac{e^{-i\lambda |w|}-1}{\lambda |w|}-\frac{\lambda \chi'(\lambda)}{i|z|+i|w|}\frac{e^{i\lambda |w|}-1}{\lambda |w|}-\frac{2i\chi'(\lambda)}{|z|^2-|w|^2}\right)\\
&\quad +\frac{e^{-\lambda|z|}}{|z|}\left(\frac{\lambda \chi'(\lambda)}{-|z|+i|w|}\frac{e^{i\lambda|w|}-1}{\lambda |w|}+\frac{\lambda \chi'(\lambda)}{|z|+i|w|}\frac{e^{-i\lambda|w|}-1}{\lambda |w|}-\frac{2i\chi'(\lambda)}{|z|^2+|w|^2}\right).
\end{align*}
Since for any $\ell\ge0$ there exists $C_\ell$ such that for any $\lambda>0$ and $w$ with $|w|\le 1/2$,
$$
\left|\partial_\lambda^\ell\left(\frac{e^{\pm i\lambda|w|}-1}{\lambda |w|}\right)\right|\le C_\ell,\quad \ell=0,1,2,...,
$$
with the bound $|z|\ge 1/2$ under the restrictions $||z|-|w||\ge1$ and $|w|\le 1/2$ at hand, we obtain  the second estimate in  \eqref{K_p9} by integrating by parts $N$ times in \eqref{Psi_2}. Changing the role of $z$ and $w$, we also obtain the third estimate in \eqref{K_p9} by the same argument.
\end{proof}

\begin{proof}[Proof of Proposition \ref{proposition_example_1}]
By Lemma \ref{K_p5}, $T_\mathbb{F}\in \mathbb B(L^p(\R^3))$ for all $1\le p\le \infty$. To disprove the $L^1$- and $L^\infty$- boundedness of $T_{K_P}$, it thus is enough to prove $T_{\mathbb{G}}\notin \mathbb B(L^1(\R^3))\cup \mathbb B(L^\infty(\R^3))$. Let
$$\Phi(u_1,u_2, x)=\int_{|y|\le R}\frac{|x-u_1|\ \chi_{\{||x-u_1|-|y-u_2||\ge 1\}}}{|x-u_1|^4-|y-u_2|^4}dy$$
be such that
\begin{align}
\label{T_G}
T_{\mathbb{G}}f_R(x)=\frac{-1-i}{4\pi\|V\|^2_{L^1}} \int_{\R^6}|V(u_1)V(u_2)|\Phi(u_1,u_2, x)du_1du_2. 
\end{align}
\vskip0.2cm
(i) \underline{The unboundedness of $T_{\mathbb{G}}$ on $L^\infty$.}
\vskip0.2cm
Suppose $R\ge1$ and $R+2R_0+1\le |x|\le R+2R_0+2$. We shall claim that
\begin{align}
\label{phi}\Phi(u_1,u_2, x)\ge \frac{\pi}{2}\ln\Big(1+\frac{R-R_0}{2R_0+1}\Big),\end{align}
uniformly for $u_1,u_2\in B(0,R_0)$ if $R$ is large enough. If  \eqref{phi} holds, then by \eqref{T_G} we obtain
$$|T_{\mathbb{G}}f_R(x)|=\frac{1}{2\sqrt{2}\pi\|V\|^2_{L^1}}\int|V(u_1)V(u_2)| \Phi(u_1,u_2, x)du_1du_2\ge \frac{1}{4\sqrt{2}}\ln\Big(1+\frac{R-R_0}{2R_0+1}\Big).$$
This implies $\|T_{\mathbb{G}}f_R\|_{L^\infty}\to \infty$ as $R\to \infty$ and thus $T_{\mathbb{G}}\notin \mathbb B(L^\infty(\R^3))$ since $\|f_R\|_{L^\infty}=1$. 
 
To prove \eqref{phi}, we let $|u_1|\le R_0$, $|u_2|\le R_0$ and set $z=y-u_2$ so that
$$\Phi(u_1,u_2, x)=\int_{|z+u_2|\le R}\frac{|x-u_1|\ \chi_{\{||x-u_1|-|z||\ge 1\}}}{|x-u_1|^4-|z|^4}dz.$$
Since $|x-u_1|\ge |x|-|u_1|\ge R+R_0+1\ge |z|+1$ if $|z+u_2|\le R$, we have $\chi_{\{||x-u_1|-|z||\ge 1\}}=1$ and
\begin{align}
\label{phi1}
\Phi(u_1,u_2, x)&=\int_{|z+u_2|\le R}\frac{|x-u_1|}{|x-u_1|^4-|z|^4}dz\nonumber\\
&\ge \int_{|z|\le R-R_0}\frac{|x-u_1|}{|x-u_1|^4-|z|^4}dz\nonumber\\
&= 4\pi\int_0^{R-R_0}\frac{|x-u_1|r^2}{|x-u_1|^4-r^4}dr\nonumber\\
&= 2\pi\int_0^{R-R_0}|x-u_1|\Big(\frac{1}{|x-u_1|^2-r^2}-\frac{1}{|x-u_1|^2+r^2}\Big)dr\nonumber\\
&\ge 2\pi \Big(\int_0^{\frac{R-R_0}{|x-u_1|}}\frac{1}{1-s^2}ds\Big)-2\pi\Big(\int_\R\frac{1}{1+s^2}ds\Big)\nonumber\\
&=\pi\ln\Big(1+\frac{2(R-R_0)}{|x-u_1|-R+R_0}\Big)-2\pi^2.
\end{align}
Since $|x-u_1|-R+R_0\le 4R_0+2$, by the monotonicity of $\ln \big(1+\frac{1}{x}\big)$, \eqref{phi1} hence implies \eqref{phi} for sufficiently large $R$.  \vskip0.2cm
(ii)  \underline{The unboundedness of $T_{\mathbb{G}}$ on $L^1$.}
\vskip0.2cm
Let $|x|\ge 3R_0+2$, $|u_1|\le R_0$, $|u_2|\le R_0 $ and $|y|\le 1$. Then
 $$|x-u_1|\ge  |x|-|u_1|\ge 2R_0+2\ge 2|y-u_2|,$$
which implies
$$|x-u_1|-|y-u_2|\ge \frac{1}{2}|x-u_1|\ge R_0+1\ge 1.$$
Hence when $|x|\ge 3R_0+2$,
\begin{align*}
|T_{\mathbb{G}}f_1(x)|=\frac{1}{2\sqrt{2}\pi\|V\|^2_{L^1}}\int_{\R^6}|V(u_1)V(u_2)|\Big(\int_{|y|\le 1}\frac{|x-u_1|}{|x-u_1|^4-|y-u_2|^4}dy\Big)du_1du_2
\end{align*}
Since $|x-u_1|^4-|y-u_2|^4\le |x-u_1|^4$ and $|x-u_1|\le |x|+|u_1|\le \frac{4}{3}|x|$, this shows
\begin{align*}
\int_{|x|\le R}|T_\mathbb{G}f_1|dx
&\gtrsim\frac{1}{\|V\|^2_{L^1}}\int_{\R^6}|V(u_1)V(u_2)|
\Big(\int_{3R_0+2\le |x|\le R}\int_{|y|\le 1}\frac{|x-u_1|}{|x-u_1|^4-|y-u_2|^4}dydx\Big)du_1du_2\\
&\gtrsim\frac{1}{\|V\|^2_{L^1}}\int_{\R^6}\int_{3R_0+2\le |x|\le R}\frac{|V(u_1)V(u_2)|}{|x-u_1|^3}dx\ du_1du_2\\
&\gtrsim\int_{3R_0+2\le |x|\le R}\frac{1}{|x|^3}dx
\gtrsim \ln\Big(\frac{R}{3R_0+2}\Big)\to \infty
\end{align*}
as $R\rightarrow\infty$. Therefore, $T_{\mathbb{G}}f_1\notin L^1(\R^3)$ and $T_{\mathbb G}\notin \mathbb B(L^1(\R^3))$ since $f_1\in L^1(\R^3)$. 
\end{proof}

\section{Appendix--the proof of  Lemma \ref{lemma_3_3}}
\label{appendix}

In the appendix,  we will prove Lemma \ref{lemma_3_3} on the  expansion of $M^{-1}(\lambda)$ near  $\lambda=0$ in regular case. Before the proof, we list the following  lemma used in the proof.
\begin{lemma}\label{lemma-JN}(\cite[Lemma 2.1]{JN})
	Let $A$ be a closed operator and $S$ be a projection. Suppose $A+S$ has a bounded inverse. Then $A$ has
	a bounded inverse if and only if
	\begin{equation*}		
		a:= S-S(A+S)^{-1}S
	\end{equation*}
	has a bounded inverse in $SH$, and in this case
	\begin{equation*}		
		A^{-1}= (A+S)^{-1} + (A+S)^{-1}S a^{-1} S(A+S)^{-1}.
	\end{equation*}	
\end{lemma}

\vskip0.2cm
\begin{proof}[The proof of Lemma \ref{lemma_3_3}].
	Firstly, we  expand  $M(\lambda)$ as follows for  small $\lambda$ by  Taylor expanding the exponentials in $F_+(\lambda|x-y|)=(\lambda|x-y|)^{-1}(e^{i(\lambda|x-y|)}-e^{-(\lambda|x-y|)})$:
	\begin{equation}\label{Mpm-2}
		\begin{split}
			M(\lambda)=U+vR^+_0(\lambda^4)v= \frac{a}{\lambda}P +T+a_1\lambda vG_1v+O\big(\lambda^3v(x)|x-y|^4v(y)\big),
		\end{split}
	\end{equation}
	where
	\begin{equation*}
		\begin{split}
			&T=U+vG_0v,\ \  G_0=-\frac{|x-y|}{8\pi},\ \ G_1(x,y)=|x-y|^2,\\
			& a= \frac{1+i}{8\pi} \|V \|_{L^1},\ \ a_1=\frac{1- i}{8\pi \cdot3!}, \ \ v=\sqrt{|V|},
		\end{split}
	\end{equation*}
and $O\big(\lambda^3v(x)|x-y|^4v(y)\big)$ denotes a $\lambda$-dependent absolutely bounded operator whose kernel is dominated by $ C\lambda^3v(x)|x-y|^4v(y)$ for some $C>0$.
	Next,  we are devoted to obtaining \eqref{lemma_3_3_1}.   Write
	\begin{equation*}
		\begin{split}
			M(\lambda)= \frac{a}{\lambda}\big(P +\frac{\lambda}{a}T+\frac{a_1}{a}\lambda^2 vG_1v+O\big(\lambda^4v(x)|x-y|^4v(y)\big)\big):=\frac{a}{\lambda}\widetilde{M}(\lambda).
		\end{split}
	\end{equation*}
Clearly, it suffices to establish the inverse of $\widetilde{M}(\lambda)$ in order to obtain $M^{-1}(\lambda)$ for small $ \lambda$. For convenience, in the following, we also use $O(\lambda^k)$ as an absolutely bounded operator on $L^2(\mathbb{R}^3)$, whose bound is dominated by $C\lambda^k$.

	Note that by Neumann series expansion, the operator $\widetilde{M}(\lambda)+Q$ is inverse for $ \lambda$ sufficiently small,  and its inverse operator is given by
	\begin{equation*}
		\begin{split}
			\big(\widetilde{M}(\lambda)+Q\big)^{-1}
			= I-\sum_{k=1}^3 \lambda^kB_k+O(\lambda^4),
		\end{split}
	\end{equation*}
	where $B_k(1\leq k\leq 3)$  are absolutely bounded operators in $L^2(\mathbb{R}^3)$ as follows:
	\begin{equation*}
		\begin{split}
			B_1=& \frac{1}{a} T, \,
			B_2= \frac{a_1}{a} vG_1v -\frac{1}{a^2}T^2, \,
			B_3= -\frac{a_1}{a^2} (TvG_1v +vG_1vT )
			+ \frac{1}{a^3}T^3.
		\end{split}
	\end{equation*}
	Let
	\begin{equation*}
		\begin{split}
			M_1(\lambda)
			&:=Q-Q(\widetilde{M}(\lambda)+Q)^{-1}Q\\
			&=\frac{\lambda}{a}\big(QTQ+a\lambda QB_2 Q+a \lambda^2QB_3 Q +O(\lambda^3)\big):=\frac{\lambda}{a}\widetilde{M}_1(\lambda).
		\end{split}
	\end{equation*}
	Since zero is a regular point  of $H$, i.e.  $QTQ$ is invertible on $QL^2(\R^3)$,   then $\widetilde{M}_1(\lambda)$ is  invertible on $QL^2(\mathbb{R}^3)$. By Neumann series expansion,  as $ \lambda$ sufficiently small,  one has on $QL^2(\mathbb{R}^3)$:
	\begin{equation*}
		\begin{split}
			M_1^{-1}(\lambda)=\frac{a}{\lambda} \widetilde{M}_1^{-1}(\lambda)=\frac{a}{\lambda}D_0-a^2 D_0B_2 D_0+\lambda\big(a^3D_0(B_2 D_0)^2-a^2D_0B_3D_0\big)+O(\lambda^2),
		\end{split}
	\end{equation*}
	where $D_0:= (QTQ)^{-1}$.  Thus according to Lemma \ref{lemma-JN}, the inverse operator $\widetilde{M}^{-1}(\lambda)$ exists for sufficiently small $\lambda$, and
	$$
\widetilde{M}^{-1}(\lambda)=	\big(\widetilde{M}(\lambda)+Q\big)^{-1} +	\big(\widetilde{M}(\lambda)+Q\big)^{-1} QM_1^{-1}(\lambda)Q 	\big(\widetilde{M}(\lambda)+Q\big)^{-1}.
	$$
	Hence we finally obtain  as sufficiently small $\lambda$,
	\begin{equation*}
		\begin{split}
			M^{-1}(\lambda)&=\frac{\lambda}{a}\widetilde{M}^{-1}(\lambda)=D_0+\lambda\big(\frac{1}{a}Q-\frac{1}{a}D_0T-\frac{1}{a}TD_0+\frac{1}{a}D_0T^2D_0-a_1 D_0vG_1vD_0\big)\\
			& \ \ \ \ \ \ \ \ \ \ \ \ \ \ \ + \frac{1}{a}\lambda P+\lambda^2A_2+O(\lambda^3)\\
			& :=QA_{0}Q+\lambda\left(QA_{1,0}+A_{0,1}Q\right)
			+\lambda\tilde P+\lambda^2A_2+\Gamma_3(\lambda),
		\end{split}
	\end{equation*}
where  $A_0, A_{1,0}, A_{0,1}$ and $A_2$ are absolutely  bounded operators on $L^2(\R^3)$ independent of $\lambda$, and the error term $\Gamma_3(\lambda)$  satisfies the desired bounds \eqref{lemma_3_3_4}.
%	\begin{equation*}
%		\begin{split}
%			A_2=&-D_0B_2+ a D_0B_2D_0B_1+a^2D_0(B_2D_0)^2-a D_0B_3D_0\\
%			& +B_1 D_0B_1+a B_1D_0B_2 D_0-B_2D_0-\frac{1}{a}B_1.
%		\end{split}
%	\end{equation*}
So we complete the whole proof.
\end{proof}

\section*{Acknowledgments}
H. Mizutani is partially supported by JSPS KAKENHI Grant-in-Aid for Scientific Research (C) \#JP21K03325. Z. Wan and X. Yao are partially supported by NSFC grants No.11771165 and 12171182. The authors would like to express their thanks to Professor Avy Soffer for his interests and insightful discussions about topics on higher-order operators. Finally, we also thank the anonymous referee for insightful and helpful comments which greatly improve the present version.

%%%%%%%%%% Bibliography %%%%%%%%%%%%%%%%%%%

\end{document}